\documentclass[11pt,reqno,twosdie]{amsart}

\usepackage{amsmath,amssymb,mathrsfs}

\usepackage[asymmetric,top=2.5cm,bottom=2.8cm,left=2.38cm,right=2.38cm]{geometry}
\usepackage{amsmath,amsfonts,amsthm,mathrsfs,amssymb,cite}
\usepackage[usenames]{color}

\usepackage{mathtools} 
\usepackage{graphics}
\usepackage{framed}

\usepackage{enumerate}

\newcommand{\mm}[1]{{\color{black}{#1}}}

\usepackage{graphicx}
\usepackage{latexsym} 
\usepackage{enumerate}

\theoremstyle{plain}
\newtheorem{theorem}{Theorem}[section]
\newtheorem{lemma}[theorem]{Lemma}
\newtheorem{corollary}[theorem]{Corollary}

\theoremstyle{remark}
\newtheorem{definition}[theorem]{Definition}
\newtheorem{remark}[theorem]{Remark}

\renewcommand{\eqref}[1]{\textnormal{(\ref{#1})}}

\numberwithin{equation}{section}

\newcommand{\bfx}{\mathbf{x}}

\newcommand{\R}{\mathbb{R}}

\newcommand{\N}{\mathbb{N}}

\def\bsi{{\mathrm{i}}}

\def\bsl{ \boldsymbol {l} }
 
\def\bfx{ \mathbf {x} }
 
\usepackage{color}

\title[Unique identifiability for inverse obstacle scattering problems]{Two unique Identifiability results for  inverse scattering  problems within polyhedral geometries} 

\author{Xinlin Cao}
\address{Johann Radon Institute for Computational and Applied Mathematics, Austrian Academy of Sciences, Altenbergstr. 69 A--4040 Linz, Austria.}
\email{xlcao.pdeip@gmail.com, xinlin.cao@oeaw.ac.at}

\author{Huaian Diao}
\address{School of Mathematics, Jilin University,
Changchun, Jilin 130012, China.}
\email{diao@jlu.edu.cn}

\author{Hongyu Liu}
\address{Department of Mathematics, City University of Hong Kong, Kowloon, Hong Kong, China.}
\email{hongyu.liuip@gmail.com, hongyliu@cityu.edu.hk}

\author{Jun Zou}
\address{Department of Mathematics, The Chinese University of Hong Kong, Hong Kong, China}
\email{zou@math.cuhk.edu.hk}

\date{}

\begin{document}

\begin{abstract}

We consider the unique determinations of impenetrable obstacles or diffraction grating profiles in $\mathbb{R}^3$ by a single far-field measurement within polyhedral geometries. We are particularly interested in the case that the scattering objects are of impedance type. We derive two new unique identifiability results for the inverse scattering problem in the aforementioned two challenging setups. The main technical idea is to exploit certain quantitative geometric properties of the Laplacian eigenfunctions which were initiated in our recent works \cite{CDL2,CDL3}. In this paper, we derive novel geometric properties that generalize and extend the related results in \cite{CDL3}, which further enable us to establish the new unique identifiability results. It is pointed out that in addition to the shape of the obstacle or the grating profile, we can simultaneously recover the boundary impedance parameters. 

\medskip 
\noindent{\bf Keywords} Unique identifiability, inverse obstacle scattering, inverse grating, single far-field pattern, Laplacian eigenfunction, geometric structure,
 
\medskip
\noindent{\bf Mathematics Subject Classification (2010)}: 35P05, 35P25, 35R30, 35Q60

\end{abstract}

\maketitle

\section{Introduction}\label{sec:Intro}


\subsection{Mathematical setup and main results for the inverse obstacle problem.}\label{subsec-obstacle}

Let $\Omega\subset\mathbb{R}^3$ be a bounded Lipschitz domain such that $\mathbb{R}^3\backslash\overline{\Omega}$ is connected. 
Let $u^i$ be an incident plane field of the form
\begin{equation}\label{ui0}
u^i:=u^i(\mathbf{ x};k,\mathbf{d})=e^{\bsi k\mathbf{ x}\cdot\mathbf{d}},\quad \mathbf x\in\mathbb{R}^3,
\end{equation}
where $k\in\mathbb{R}_+$ denotes the wavenumber and $\mathbf{d}\in\mathbb{S}^2$ signifies the incident direction. Denote $u^s$ as the scattered wave field generated from the interruption of the propagation of $u^i$ by an impedance obstacle.
Define  $u:=u^i+u^s$ to be the total wave field. The forward scattering problem is described as follows:
\begin{equation}\label{forward0}
\begin{cases}
& \Delta u + k^2 u = 0\qquad\quad \mbox{in }\ \ \mathbb{R}^3\backslash\overline{\Omega},\medskip\\
& u =u^i+u^s\hspace*{1.56cm}\mbox{in }\ \ \mathbb{R}^3,\medskip\\
& \partial_\nu u+\eta u=0\hspace*{1.3cm}\mbox{on}\ \ \partial\Omega,\medskip\\
&\displaystyle{ \lim_{r\rightarrow\infty}r^{\frac{1}{2}}\left(\frac{\partial u^{s}}{\partial r}-\mathrm{i}ku^{s}\right) =\,0,\quad r:=|x|,}
\end{cases}
\end{equation}
where the last limit is known as the Sommerfeld radiation condition that holds uniformly in $\hat{\mathbf{ x}}:=\mathbf{ x}/|\mathbf{ x}|\in\mathbb{S}^2$. 
The Robin type boundary condition is referred to as the impedance boundary condition for an impedance obstacle, where $\nu$ denotes the exterior unit normal vector to $\partial\Omega$ and $\eta\in L^\infty(\partial\Omega)$ represents the corresponding boundary impedance parameter. If $\eta  \neq 0$, $\Omega$ is  said to be an impedance obstacle. 

The wellposedness of the forward scattering problem \eqref{forward0} is known \cite{CK, Mclean} and there exists a unique solution $u\in H^1_{loc}(\mathbb{R}^3\backslash\overline{\Omega})$ fulfilling the following expansion, which holds uniformly for $\hat{\mathbf x}$:
\begin{equation}\label{eq:far}
u^s(\mathbf x;k, \mathbf d)
=\frac{e^{\mathrm{i}kr}}{r^{{1/2}}} u_{\infty}(\hat{\mathbf x};k,\mathbf{d})+\mathcal{O}\left(\frac{1}{r^{3/2}}\right)\quad\mbox{as }\,r\rightarrow\infty,
\end{equation}
where $u_\infty$ is known as the associated far-field pattern or the scattering amplitude. 

The corresponding inverse obstacle scattering problem to \eqref{forward0} is to recover the shape of the obstacle $\Omega$ as well as the associated impedance parameter $\eta$ by knowledge of the far-field pattern $u_\infty(\hat{\mathbf{ x}};k, \mathbf{d})$. By introducing an abstract operator $\mathcal{F}$ which sends the obstacle to the corresponding far-field pattern, the aforementioned inverse problem can be formulated as
\begin{equation}\label{inverse}
\mathcal{F}(\Omega, \eta)=u_\infty(\hat{\mathbf{ x}};k, \mathbf{d}),
\end{equation} 
which is nonlinear and ill-posed (cf.\cite{CK}).

%


We first introduce the concept of ``admissible" obstacles by presenting certain a-priori conditions on the underlying obstacles in deriving our main unique determination results for the inverse obstacle problem.

\begin{definition}\label{ad obstacle0}
	Let $\Omega\subset\mathbb{R}^3$ be an open polyhedron associated with the boundary impedance condition (i.e. the third  equation) in \eqref{forward0}. $\Omega$ is said to be an \emph{admissible polyhedral obstacle}, if 
	for any  face $\Sigma_m\subset\partial\Omega$ of $\Omega$, the corresponding surface impedance parameter $\eta$ on $\Sigma_m$ is  a real-analytic function with the form 
	\begin{equation}\label{vari-eta0}
	\eta(\mathbf x)\big  |_{ \Sigma_m }=\eta(r,\theta,\phi)\big |_{ \Sigma_m}=\alpha_0+ \sum_{\ell=1}^{\infty}\alpha_{\ell}^{(m)}(\theta, \phi)r^{\ell}, \,  \mathbf x=(r\sin\theta\cos\phi,r\sin\theta\sin\phi,r\cos\theta)\in \Sigma_m,
	\end{equation}
	fulfilling that the complex constant $\alpha_0 \neq 0$,  where $\partial \Omega=\cup_{m=1}^p \Sigma_m$ and $\Sigma_m$  is a face of $\Omega$. 
\end{definition}

\begin{remark}
	It is noted that in Definition \ref{ad obstacle0}, we assume that the constant part  in the expansion \eqref{vari-eta0} of the  impedance parameter on each face of $\Omega$ is a fixed constant  $\alpha_0$. Furthermore,  if $\alpha_{\ell}^{(m)} \equiv 0 $ for $m\in {1,\ldots,  p}$ and $\ell\in \mathbb N$, it is seen that the admissible polyhedral obstacle $\Omega$  has the same boundary impedance parameter on each face of $\Omega$. Compared with the admissible polyhedral obstacles introduced  in \cite[Definition 6.1]{CDL3}, we completely remove the root assumption of the associated Legendre polynomials on the dihedral angle between any two adjacent faces of the underlying polyhedron, which may not be easily verified in practical applications. Therefore, our new unique identifiability result in what follows for determining the admissible impedance obstacle as introduced above can be applied to more general scenarios in the inverse obstacle scattering problem. 
\end{remark} 


In the subsequent discussion, for notational unification, we write an admissible polyhedral obstacle as ($\Omega, \eta$). We also define the admissible complex polyhedral obstacles as follows:


\begin{definition}\label{def60}
	$\Omega$ is said to be an admissible complex polyhedral obstacle if it consists of finitely many pairwise disjoint admissible polyhedral obstacles.
	That is,
	\begin{equation*}
	(\Omega, \eta)=\bigcup_{j=1}^l (\Omega_j, \eta_j),\quad \eta=\sum_{j=1}^l \eta_j\chi_{\partial\Omega_j\cap\partial\Omega}
	\end{equation*}
	where $l\in\mathbb{N}$ and each $(\Omega_j, \eta_j)$ is an admissible polyhedral obstacle.
	
\end{definition}

Following the notations in \cite{CDL3}, we let $\Pi_{1}, \Pi_2$ be two adjacent faces of a polyhedron $\Omega$. Denote $\mathcal{E}(\Pi_{1}, \Pi_2, \bsl)$ to be an edge corner associated with $\Pi_{1}$ and $\Pi_2$. $\mathcal{V}(\{\Pi_\ell\}_{\ell=1}^n, \mathbf{ x}_0)$ signifies the vertex corner formulated by $\Pi_1, \Pi_2, \cdots, \Pi_n$ at the vertex $\mathbf{ x}_0\in \partial \Omega$. It is clear that a vertex corner  $\mathcal{V}(\{\Pi_\ell\}_{\ell=1}^n, \mathbf{ x}_0)$ is composed of finite many edge corners, which intersect at $\mathbf{ x}_0$. Now, recall the definitions for rational and irrational obstacles based on the concept of the rational and irrational corners introduced in \cite{CDL3} as follows.


\begin{definition}\cite[Definition 4.1]{CDL3}
	{\color{black}{
	Let ${\mathcal E}(\Pi_1, \Pi_2,\bsl)$ be an edge corner associated with $\Pi_{1}$ and $\Pi_2$. Denote the dihedral angle of $\Pi_1$ and $\Pi_2$ by $\phi=\alpha \cdot \pi $, $\alpha\in(0,1)$. }}
	If $\phi$ is an irrational dihedral angle, namely, $\alpha$ is an irrational number, then ${\mathcal E}(\Pi_1, \Pi_2,\bsl)$  is said to be an {\it irrational} edge corner. Otherwise it is said to be  a {\it rational} edge corner. For a rational edge corner ${\mathcal E}(\Pi_1, \Pi_2,\bsl)$,  it is called a rational angle of degree $p$ of ${\mathcal E}(\Pi_1, \Pi_2,\bsl)$ if $\alpha = q/p$ with $p,q\in \mathbb N$ and irreducible.
	

\end{definition}

\begin{definition}\cite[Definition 4.2]{CDL3}\label{def:irration}
	Let ${\mathcal V}(\{\Pi_\ell\}_{\ell=1}^n,\bfx_0)$ be a vertex corner, where $n\in \N$ and $n\geq 3$. It is clear that ${\mathcal V}(\{\Pi_\ell\}_{\ell=1}^n,\bfx_0)$ is composed of the following $n$ edge corners:
	$$
	{\mathcal E}_\ell:=	{\mathcal E}(\Pi_\ell, \Pi_{\ell+1},\bsl_\ell), \quad {\mathcal E}_n:={\mathcal E}(\Pi_n, \Pi_{1},\bsl_n), \quad \Pi_{n+1}:=\Pi_1, \quad \ell=1,2,\ldots, n-1, 
	$$
	where $\bsl_\ell$ is the line segment of $\Pi_\ell \cap \Pi_{\ell+1}$ and $\bsl_n$ is the line segment of  $\Pi_n \cap \Pi_{1}$, respectively.  Denote
	\begin{equation}\label{eq:ir index}
	\begin{split}
	I_{\sf IR}&=\{\ell \in \N~|~1\leq \ell\leq n,\quad {\mathcal E}_\ell \mbox{ is an irrational edge corner}\},\\
	I_{\sf R}&=\{\ell \in \N~|~1\leq \ell\leq n,\quad {\mathcal E}_\ell \mbox{ is a rational edge corner}\}.  
	\end{split}
	\end{equation}
	If $\#I_{\sf IR}\geq 1$, then ${\mathcal V}(\{\Pi_\ell\}_{\ell=1}^n,\bfx_0)$ is said to be an {\it irrational } vertex corner. If $\#I_{\sf IR}\equiv 0$, then ${\mathcal V}(\{\Pi_\ell\}_{\ell=1}^n,\bfx_0)$ is said to be a {\it rational } vertex corner. For a { rational } vertex corner ${\mathcal V}(\{\Pi_\ell\}_{\ell=1}^n,\bfx_0)$ composed of edge corners ${\mathcal E}_\ell:={\mathcal E}(\Pi_\ell, \Pi_{\ell+1},\bsl_\ell)$, the largest degree of ${\mathcal E}_\ell\, (\ell=1,\ldots, n)$ is referred to as the \emph{rational degree} of ${\mathcal V}(\{\Pi_\ell\}_{\ell=1}^n,\bfx_0)$. 
\end{definition}

Next, we provide the definition for rational and irrational admissible obstacles.

\begin{definition}\label{ir obstacle}
	Let $(\Omega, \eta)$ be an admissible polyhedral obstacle. If \mm{there exists}  a rational vertex corner, then it is said to be a \emph{rational obstacle}. If all the vertex corners of $\Omega$ are irrational, then it is called an \emph{irrational obstacle}. The smallest degree of the rational corner of $\Omega$ is referred to as the \emph{rational degree} of $\Omega$, which is denoted by ${\sf deg}(\Omega )$. 
\end{definition}

With all the necessary notations and definitions introduced above, we are now in a position to give the local unique identifiability results for an admissible complex polyhedral obstacle by a single far-field measurement with respect to rational and irrational cases, separately.

\begin{theorem}\label{inverse10}
	Considering the scattering problem \eqref{forward0} associated with two admissible irrational complex polyhedral obstacles $(\Omega_j, \eta_j)$ in $\R^3$, $j=1,2$. Let  $u_\infty^j(\hat{\mathbf x}; k, \mathbf{d})$ be the corresponding far-field patterns associated with $(\Omega_j, \eta_j)$ and the incident wave $u^i$ defined in \eqref{ui0}. Let $\mathbf{G}$ be the unbounded connected component of $\mathbb{R}^3\backslash\overline{(\Omega_1\cup\Omega_2)}$.
	Suppose that  
	\begin{equation}\label{eq:cond10}
	u_\infty^1(\hat {\mathbf x}; k, \mathbf{d})= u_\infty^2(\hat {\mathbf x}; k, \mathbf{d}), \ \ 
	\mbox{for } ~~ \mbox{all} ~~\hat{\mathbf x}\in\mathbb{S}^2,
	\end{equation}
	then 
	$
	(\partial \Omega_1 \backslash \partial \overline{ \Omega }_2 )\bigcup (\partial \Omega_2 \backslash \partial  \overline{ \Omega }_1 )
	$
	cannot possess an edge  corner on $\partial \mathbf{G}$. 
	
Moreover,
	\begin{equation}\label{eta0}
	\eta_1=\eta_2\quad\mbox{on}\quad \partial\Omega_1\cap\partial{\Omega_2}.
	\end{equation}
\end{theorem}


For any vertex $\mathbf{ x}_c$ of a polyhedron $\Omega\in\mathbb{R}^3$, we denote for $r\in\mathbb{R}_+$ that 
$
\Omega_r(\mathbf x_c)=B_r(\mathbf x_c)\cap \mathbb{R}^3\backslash\overline{\Omega}.
$
Define 
\begin{equation}\label{eq:l1}
\mathcal{L}(f)(\mathbf x_c):=\lim_{r\rightarrow+0}\frac{1}{|\Omega_r(\mathbf x_c)|}\int_{\Omega_r(\mathbf x_c)} f(\mathbf x)\ {\rm d} \mathbf x
\end{equation}
if the limit exists for any $f\in L_{loc}^2(\mathbb{R}^3\backslash\overline{\Omega})$. It is easy to see that if $f(\mathbf x)$ is continuous in $\overline{\Omega_{\epsilon_0}(\mathbf x_c)}$ for a sufficiently small $\epsilon_0\in\mathbb{R}_+$, then $\mathcal{L}(f)(\mathbf x_c)=f(\mathbf x_c)$. According to \cite[Remark 6.11]{CDL3}, \eqref{eq:l1} can  be fulfilled in certain practical scenarios. 

Using the technical assumption  \eqref{eq:l1}, the unique determination of rational obstacles can be stated as follows. 

\begin{theorem}\label{inverse2}
	For a fixed $k\in\mathbb{R}_+$, let $(\Omega_j, \eta_j)$, $j=1,2$, be 
	two admissible complex rational obstacles,
	with $u^j_{\infty}(\hat{\mathbf x};k,\mathbf{d})$ being their corresponding far-field patterns associated with 
	the incident field $u^i$ defined in \eqref{ui0}, where $\mathbf d$ is the incident direction. Assume that 
	\begin{align}\label{eq:deg cond}
		{\sf deg}(\Omega_j ) \geq 3, \quad j=1,2,
	\end{align}
	where ${\sf deg}(\Omega_j )$ is the  rational degree of $\Omega_j$  defined in Definition \ref{ir obstacle}. 
	We further write $\mathbf{G}$ for the unbounded connected component of $\mathbb{R}^3\backslash\overline{(\Omega_1\cup\Omega_2)}$. 
	Then if 
	\begin{equation}\label{cond51}
	u_{\infty}^1(\hat {\mathbf x}; k, \mathbf{d} )= u_{\infty}^2(\hat {\mathbf x}; k, \mathbf{d}), \ \ \hat{\mathbf x}\in\mathbb{S}^2,\quad\mbox{and}\quad
	\mathcal{L}(\nabla u^j)(\mathbf x_c)\neq 0,\ j=1,2,
	\end{equation}
	for all vertices $\mathbf x_c$ of $\Omega$, the set 
	$
	(\partial \Omega_1 \backslash \partial \overline{ \Omega_2 } )\cup 
	(\partial {\Omega_2 } \backslash \partial\overline{  \Omega_1}  )
	$
	can not possess an  edge corner on  $\partial \mathbf{G}$.
	\mm{Moreover,
		\begin{equation}\notag
		\eta_1={\eta_2}\quad\mbox{on}\quad \partial\Omega_1\cap\partial{\Omega_2}.
		\end{equation}}
\end{theorem}

Theorems \ref{inverse10} and \ref{inverse2} actually reveal the unique determination for a certain kind of impedance obstacles locally in the neighborhood of the corner. It is easy to verify that if the underlying admissible complex obstacles are convex, then there holds the global uniqueness results accordingly.  Indeed, one has
\begin{corollary}\label{coro-inver}
For a fixed $k\in\mathbb{R}_+$,  let $(\Omega, \eta)$ and $(\widetilde\Omega, \widetilde\eta)$ be 
	two convex admissible irrational complex polyhedral obstacles, with $u^j_{\infty}(\hat{\mathbf x};k,\mathbf{d})$, $j=1,2$, being their corresponding far-field patterns associated with 
	the incident field $u^i$ defined in \eqref{ui0}, where $\mathbf d$ is the incident direction. If \eqref{eq:cond10} is fulfilled, 
	then
	\begin{equation}\notag
		\Omega=\widetilde{\Omega },\quad \eta=\widetilde{\eta}. 
	\end{equation}

\end{corollary}

Compared with the unique identifiability study in \cite[Section 6]{CDL3}, the results presented in this paper significantly relax the number of the measurements from ``at most two far-field patterns" to ``only one far-field pattern". This is made possible by relaxing the technical condition $u(\mathbf{0})=0$ in the vanishing properties of Laplacian eigenfunctions in $\R^3$, which will be systematically investigated in the subsequent sections. Moreover, the a-priori conditions imposed on ``admissible obstacles" are also relaxed by removing the conditions on the associated Legendre polynomials. This is derived from the fact that the vanishing orders of the Laplacian eigenfunction $u$ at an edge corner can be determined by corresponding two intersecting adjacent planes without any a-prior information of the other planes; see \cite[Theorem 3.1]{CDL3} for more relevant discussions.

\subsection{Mathematical setup and main results for the inverse diffraction grating problem}

First, we give a brief review of the mathematical setup of the diffraction grating profile. Assume that the diffraction grating involves an impenetrable surface $\Lambda_f$ which is $2\pi$-periodic with respect to $\mathbf{ x}':=(x_1, x_2)$. Precisely speaking, denote
\begin{equation}\label{def-grating0}
\Lambda_f=\{\mathbf{ x}:=(\mathbf{ x}', x_3)\in\R^3; x_3=f(\mathbf{ x}')\},
\end{equation}
where $f$ is a bi-periodic Lipschitz function with period $2\pi$ with respect to $x_1$ and $x_2$. Let 
\begin{equation}\label{def-space0}
\Omega_f:=\{\mathbf{ x}\in\R^3; \mathbf{ x}'\in\R^2, x_3>f(\mathbf{ x}')\}
\end{equation}
be the unbounded domain filled with an isotropic homogeneous medium. Suppose that the incident plane wave $u^i(\mathbf{ x}; k, \mathbf{d})$, where $k\in\R_+$ and
\begin{equation}\label{def-dire0}
\mathbf{d}:=\mathbf{d}(\theta, \phi)=(\sin\phi\cos\theta, \sin\phi\sin\theta, -\cos\phi),\quad \theta\in[0,2\pi),\ \phi\in(-\pi/2,\pi/2),
\end{equation}
propagates to $\Lambda_f$ from the top. The total wave field fulfills the following Helmholtz system
\begin{equation}\label{model0}
\Delta u+k^2 u=0 \ \mbox{ in } \ \Omega_f; \quad \partial_\nu u+\eta u=0 \ \mbox{ on } \ \Lambda_f,
\end{equation}
where $\eta$ denotes the surface impedance parameter. To ensure the wellposedness of \eqref{model0}, the total wave field $u$ is supposed to be ${\alpha}$-quasiperiodic with respect to $x_1$ and $x_2$, which can be defined more rigorously as 

\begin{definition}\label{alpha-quasi0}
	$u$ is said to be $\alpha$-quasiperiodic in $\mathbf{ x}'$ with respect to $x_1$ and $x_2$ if there holds
	\begin{equation}\notag
	u(\mathbf{ x}'+2\pi \mathbf{n},x_3)=e^{\bsi 2\pi {\alpha}\cdot \mathbf n}u(\mathbf{ x}',x_3)
	\end{equation}
	for any $\mathbf n=(n_1, n_2)\in\mathbb{Z}^2$.
\end{definition}

The corresponding scattered wave $u^s$ satisfies the following Rayleigh series expansion
\begin{equation}\label{rayley0}
u^s(\mathbf{ x})=\sum_{\mathbf{n}\in\mathbb{Z}^2}u_n e^{\bsi\alpha_n\cdot\mathbf{ x}'+\bsi\beta_n x_3}:=\sum_{\mathbf{n}\in\mathbb{Z}^2}u_n e^{\bsi \bf{\xi}_n\cdot\mathbf{ x}},\quad x_3>\max_{\mathbf{ x}'\in[0,2\pi)^2}f(\mathbf{ x}'),
\end{equation}
where $u_n:=u_n(k)\in\mathbb{C}$ are the Rayleigh coefficients of $u^s$ and
\begin{equation}\label{nota0}
{\bf{\xi_n}}=(\alpha_n, \beta_n),\ \mbox{with } \
\alpha_n:=\mathbf{n}+\alpha, \beta_n^2=k^2-|\alpha_n|^2,
\end{equation}
where $\Im{\beta_n}\geq0$ if $|\alpha_n|^2>k^2$. The existence and uniqueness of the $\alpha$-quasiperiodic solution to \eqref{model0} with $\eta$ being a constant fulfilling $\Im\eta>0$ can be seen in \cite{Alber}.

Define
\begin{equation}\notag
\Gamma_b:=\{(\mathbf{ x}', x_3); \mathbf{ x}'\in[0, 2\pi)^2, x_3=b\}, \quad b>\max_{\mathbf{ x}'\in[0,2\pi)^2}|f(\mathbf{ x}')|
\end{equation}

The inverse problem associated with \eqref{model0} is to determine $\Lambda_f$ from the knowledge of $u(\mathbf{ x}|_{\Gamma_b}; k, \mathbf{d})$. By introducing an abstract operator $\mathcal{F}$ which sends the information of $\Lambda_f$ to the measurement $u(\mathbf{ x}; k, \mathbf{d})$ on $\Gamma_b$, the inverse problem can be formulated as
\begin{equation}\label{inver-grating0}
\mathcal{F}(\Lambda_f, \eta)=u(\mathbf{ x}|_{\Gamma_b}; k, \mathbf{d}).
\end{equation}
The unique identifiability for the diffraction grating with impedance boundary condition by finite measurements  has been an open problem for a long time. In order to investigate this problem, similar to the study on the inverse obstacle problems, we first propose the necessary definition of so-called ``admissible" polyhedral gratings considered in our work as follows.

\begin{definition}\label{adm-grating0}
	Let $(\Lambda_f, \eta)$ be a bi-periodic grating as described in \eqref{def-grating0}. It is said to be an admissible polyhedral diffraction grating if $\Lambda_f$ is a polyhedral Lipschitz surface in $\R^3$, consisting of a finite number of planar faces in one periodic cell $[0, 2\pi)\times[0, 2\pi)$, and
	for any  planar faces of $\Lambda_f$,  the corresponding surface impedance parameter $\eta$ is a real-analytic function with the form \eqref{vari-eta0}, where the constant part of the expansion \eqref{vari-eta0} is nonzero.
%
%
\end{definition}

Following the definitions for irrational and rational vertex corners in Definition \ref{def:irration}, we have the following concept for admissible irrational and rational polyhedral diffraction  gratings.
\begin{definition}\label{ir diff}
	Let $(\Lambda_f, \eta)$ be an admissible polyhedral diffraction  grating. If \mm{there exists}  a rational vertex corner in one period, then it is said to be a \emph{rational polyhedral   diffraction  grating}. If all the vertex corners of $\Lambda_f$ in one period are irrational, then it is called an \emph{irrational polyhedral diffraction  grating}. The smallest degree of the rational corner of $\Lambda_f$ is referred to as the \emph{rational degree} of $\Lambda_f$, which is denoted by ${\sf deg}(\Lambda_f )$. 
\end{definition} 

We would like to point out that the vertex corner considered in Definition \ref{ir diff} includes the edge corner which is intersected by two adjacent planes as a special case.


%

Now, we can provide our main unique determination results for the inverse diffraction grating problem \eqref{inver-grating0} with respect to irrational and rational structures, respectively.


\begin{theorem}\label{thm-grat0}
	Let $(\Lambda_f,\eta_f)$ and $(\Lambda_g,\eta_g)$ be two admissible irrational polyhedral diffraction gratings and $\mathbf{G}$ be the unbounded connected component of $\Omega_f\cap\Omega_g$. Let $k\in\R_+$ be fixed and $\mathbf{d}\in\mathbb{S}^2$ be the corresponding incident direction defined in \eqref{def-dire0}. Let $\Gamma_b$ be a measurement boundary given by
	\begin{equation}\notag
	\Gamma_b:=\{(\mathbf{ x}', b)\in\R^3; \mathbf{ x}'\in[0,2\pi)^2, x_3=b, b>\max\{\max_{\mathbf{ x}'\in[0,2\pi)^2}|f(\mathbf{ x}')|,\max_{\mathbf{ x}'\in[0,2\pi)^2}|g(\mathbf{ x}')|\}\}.
	\end{equation}
	Suppose that $u_f(\mathbf{ x}; k, \mathbf{d})$ and $u_g(\mathbf{ x}; k, \mathbf{d})$ are the total wave fields measured on $\Gamma_b$ associate with $(\Lambda_f, \eta_f)$ and $(\Lambda_g, \eta_g)$, respectively. If there holds that
	\begin{equation}\label{equi-cond0}
	u_f(\mathbf{ x}; k, \mathbf{d})=u_g(\mathbf{ x}; k, \mathbf{d}) \quad\mbox{for}\quad \mathbf{ x}\in\Gamma_b,
	\end{equation}
	then $\partial\mathbf{G}\backslash\partial\Lambda_f$ can not possess an edge  corner of $\Lambda_g$ and $\partial\mathbf{G}\backslash\partial\Lambda_g$ can not possess an edge corner of $\Lambda_f$. Moreover, 
	\begin{equation}\notag
		\eta_f=\eta_g\quad\mbox{on}\quad \Lambda_f\cap\Lambda_g.
	\end{equation}
\end{theorem}

Similar to Theorem \ref{inverse2}, if we assume the total wave field to \eqref{model0}  satisfies the condition \eqref{eq:l1}  on all vertices  of an admissible polyhedral diffraction  grating,  the uniqueness results for admissible rational polyhedral diffraction gratings can be stated as

\begin{theorem}\label{thm-grat2}
	Let $(\Lambda_f,\eta_f)$ and $(\Lambda_g,\eta_g)$ be two admissible rational polyhedral diffraction gratings 
	and $\mathbf{G}$ be the unbounded connected component of $\Omega_f\cap\Omega_g$. Let $k\in\R_+$ be fixed and $\mathbf{d}\in\mathbb{S}^2$ be the corresponding incident direction defined in \eqref{def-dire0}. Let $\Gamma_b$ be a measurement boundary given by
	\begin{equation}\notag
	\Gamma_b:=\{(\mathbf{ x}', b)\in\R^3; \mathbf{ x}'\in[0,2\pi)^2, x_3=b, b>\max\{\max_{\mathbf{ x}'\in[0,2\pi)^2}|f(\mathbf{ x}')|,\max_{\mathbf{ x}'\in[0,2\pi)^2}|g(\mathbf{ x}')|\}\}.
	\end{equation}
	Assume that 
	\begin{align}\label{eq:deg cond diff}
		{\sf deg}(\Lambda_f ) \geq 3 \mbox{ and } {\sf deg}(\Lambda_g ) \geq 3, 
	\end{align}
	where ${\sf deg}(\Lambda_f ) $ and  ${\sf deg}(\Lambda_g ) $ are the  rational degrees of $\Lambda_f$ and $\Lambda_g$  defined in Definition \ref{ir diff}, respectively.  
	Suppose that $u_f(\mathbf{ x}; k, \mathbf{d})$ and $u_g(\mathbf{ x}; k, \mathbf{d})$ are the total wave fields measured on $\Gamma_b$ associate with $(\Lambda_f, \eta_f)$ and $(\Lambda_g, \eta_g)$, respectively. If there holds that
	\begin{equation}\label{equi-cond2}
	u_f(\mathbf{ x}; k, \mathbf{d})=u_g(\mathbf{ x}; k, \mathbf{d}) \quad\mbox{for}\quad \mathbf{ x}\in\Gamma_b,
	\quad
\mathcal{L}(\nabla u_f)(\mathbf{ x}_c^f)\neq0,\quad \mathcal{L}(\nabla u_g)(\mathbf{ x}_c^g)\neq0
	\end{equation}
where $\mathbf{ x}_c^f$ and $\mathbf{ x}_c^g$ are arbitrary vertices  of the admissible rational polyhedral diffraction gratings $f$ and $g$, respectively, 	then $\partial\mathbf{G}\backslash\partial\Lambda_f$ can not possess an edge corner of $\Lambda_g$ and $\partial\mathbf{G}\backslash\partial\Lambda_g$ can not possess an edge  corner of $\Lambda_f$. Moreover
	\begin{equation}\notag
	\eta_f=\eta_g\quad\mbox{on}\quad \Lambda_f\cap\Lambda_g.
	\end{equation}
\end{theorem}

\subsection{Background and discussion}

Determining an impenetrable obstacle by a minimal/optimal number of scattering measurements is a long-standing problem in the inverse scattering theory. We refer to \cite{CK18,Isa2,LiuZou08} for historical accounts as well as surveys on some existing developments in the literature. This problem has been resolved if a-priori geometric conditions are imposed on the underlying obstacle, say e.g. smallness in size (compared to the wavelength), radial symmetry or polyhedral shapes. We refer to \cite{AR,CDL2,CDL3,CY,DLW20,DLW21,DLZ21,EY2,EY3,Liua,Liu3,LPRX,LRX,LX,Liu-Zou,Liu-Zou3,LiuZou08,LiuZou07,Ron1,Ron2,Ron3} and the references cited therein for rich results on this intriguing topic. In two recent papers \cite{CDL2,CDL3}, a different perspective was proposed and the uniqueness study for the inverse scattering problem is delicately connected to the geometric properties of Laplacian eigenfunctions in certain specific setups. Within such a framework, one can establish the unique determination of polyhedral obstacles of impedance type by at most a few far-field measurements that were unable to be tackled by other means developed in the previous studies. The corresponding studies have been extended to solving a large class inverse scattering problems associated with electromagnetic and elastic waves as well \cite{DLW20, DLW21, DLZ21}. As also discussed in Section 1.1, we derive in this paper novel uniqueness results which extends the related results in \cite{CDL3} by relaxing certain technical conditions. This is made possible by exploiting the geometric structures of Laplacian eigenfunctions in a less restrictive setup compared to that in \cite{CDL3}. 

There are also rich results on the unique determination of periodic structures. It is known that in general a grating profile can be uniquely identified by infinitely many quasi-periodi incident plane waves with a fixed phase-shit \cite{Am, Kir94}. In \cite{HK}, it is shown that uniqueness by a finite number of incident plane waves could be attained under the Dirichlet boundary condition if some a-priori information about the height of the grating curve is known in $\R^2$. The global uniqueness can also be established for the inverse scattering with a finite number of incident plane waves if the grating profiles are piecewise linear; see \cite{BZZ, BZZ2, ESY, EY}. Recently, in \cite{CDL2}, the unique recovery for the inverse diffraction grating with generalized impedance boundary condition (including the Dirichlet and Neumann boundary conditions) has been proved in a unified way by at most two far-field measurements under some mild assumptions. In this paper, we establish the unique identifiability of a periodical diffraction grating with impedance boundary condition in $\R^3$ by a single far-field measurement, which makes a significant progress compared with the existing results.

The rest of this paper is organized as follows. In Section \ref{sec2}, we present the extension and generalization of our results in \cite{CDL3} on the geometric structures of Laplacian eigenfunctions. Section \ref{sec3} is devoted to the proofs of the unique identifiability results for the inverse scattering problems.

\section{Geometric properties of Laplacian eigenfunctions.}\label{sec2}


In \cite{CDL3}, certain geometric properties of Laplacian eigenfunctions in $\R^3$ were investigated. Specifically, we studied the cases for edge corners and vertex corners respectively and derived a rigorous characterization of the relationship between  the analytic behaviour of Laplacian eigenfunctions at the underlying corner point and the geometric quantities of that corner. {In fact, in the edge corner case,} the vanishing order of the eigenfunction is related to the rationality of the intersecting dihedral angle, 
whereas {in the vertex corner case}, 
the vanishing order of the eigenfunction follows a more complicated manner through the roots of the Legendre polynomials. In this section, as an extension of the study in \cite{CDL3}, we establish the vanishing properties of the Laplacian eigenfunction by getting rid of the technical condition that $u(\mathbf{ x}_c)=0$ at the intersecting point as well as relaxing the technical restrictions associated with the Legendre polynomials.

Let  $\Omega$ be an open set in $\mathbb{R}^3$. Consider $u\in L^2(\Omega)$ and $\lambda\in\mathbb{R}_+$ such that
\begin{equation}\label{eq:eig}
-\Delta u=\lambda u. 
\end{equation}
{The solution $u$ to} \eqref{eq:eig} is referred to as a (generalized) Laplacian eigenfunction. 

First, we introduce some notifications for the subsequent use. Let $\Pi$ be a flat plane in $\R^3$. For any non-empty connected open subset $\Sigma\Subset\Pi$, it is said to be a cell of $\Pi$. Denote $\Pi_\Sigma$ to be the connected component of $\Pi\cap\Omega$ which contains $\Sigma$. 

\begin{definition}\label{def-gene}
	Let $\Sigma$ be a cell of $\Pi$ and $\eta\in L^\infty(\Pi)$. Consider a Laplacian eigenfunction $u$ to \eqref{eq:eig}. If $\partial_\nu u+\eta u=0$, where $\nu$ denotes the unit normal vector that is perpendicular to $\Pi$, then $\Sigma$ is said to be the generalized singular cell in $\Omega$ and $\Pi$ is called the generalized singular plane. In particular, if $\eta\equiv0$, a generalized singular plane is called a singular plane and if $\eta=\infty$, a generalized singular plane is also called a nodal plane, which fulfills that $u|_{\Pi}=0$.
\end{definition}

Next, we follow Definition 1.4 and Definition 1.5 in \cite{CDL3} to present the precise definition for vanishing orders of the Laplacian eigenfunction $u$ at a given point associated with an edge corner or a vertex corner.

\begin{definition}\label{def-vani}
	Let $u$ be a Laplacian eigenfunction to \eqref{eq:eig}. For a given point $\mathbf{ x}_0\in\Omega$, if there exists a number $N\in\N\cup\{0\}$ such that
	\begin{equation}\label{eq-vani}
		\lim_{\rho\rightarrow0+}\frac{1}{\rho^m}\int_{B_\rho(\mathbf{x}_0)}|u(\mathbf{ x})|\,{\rm d}\mathbf{ x}=0 \quad\mbox{for}\quad m=0,1,\cdots,N+2,
	\end{equation}
	where $B_\rho(\mathbf{ x}_0)$ is a ball centered at $\mathbf{ x}_0$ with radius $r\in\R_+$, we say that $u$ vanishes at $\mathbf{ x}_0$ up to the order $N$. The largest possible $N$ such that \eqref{eq-vani} holds is called the vanishing order of $u$ at $\mathbf{ x}_0$ and we denote
	\begin{equation}\notag
		{\rm Vani}(u;\mathbf{ x}_0)=N.
	\end{equation}
	If \eqref{eq-vani} holds for any $N\in\N$ at $\mathbf{ x}_0\in\Pi$, then we say that the vanishing order at $\mathbf{ x}_0$ is infinity.
	
	In particular, if \eqref{eq-vani} is fulfilled for any given point $\mathbf{ x}_0\in\bsl$ associated with an edge corner $\mathcal{E}(\Pi_{1}, \Pi_2, \bsl)\Subset\Omega$, then we say that $u$ vanishes at $\mathbf{ x}_0$ associated with the edge corner $\mathcal{E}(\Pi_{1}, \Pi_2, \bsl)$ up to the order $N$ which is denoted as
	\begin{equation}\notag
		\rm{Vani}(u; \mathbf x_0, \Pi_{1}, \Pi_2)=N.
	\end{equation}	
For a vertex corner $\mathbf{ x}_0\in \Omega$ which is intersected by $\Pi_i$, $i=1,2,...n$, the vanishing order of $u$ at $\mathbf{ x}_0$ is defined by
\begin{equation*}\label{def:3d plane}
\mathrm{Vani}(u; \mathbf{ x}_0) :=\max\big\{\max_{i=1,2,...n-1}\mathrm{Vani}(u; \bfx_0,\Pi_i, \Pi_{i+1}), \mathrm{Vani}(u; \bfx_0, \Pi_n, \Pi_1)\big\}. 
\end{equation*}
\end{definition}

With the above definitions, we are now in a position to investigate the vanishing properties of the Laplacian eigenfunction at corners intersected by at least two generalized singular planes. Since $-\Delta$ is invariant under rigid motions, we can assume that the edge corner $\mathcal{E}(\Pi_1, \Pi_2, \bsl)$ fulfills
\begin{equation}\label{notation-10}
\bsl=\big\{~\bfx=(\bfx',x_3)\in \R^3; \bfx'=0,\ x_3\in (-H, H)\big\}\Subset\Omega,
\end{equation}
for $H\in \R^+$ throughout the rest of our paper. Indeed, this indicates that the edge corner coincides with the $x_3$-axis. For simplification, we further assume that $\Pi_{1}$ coincides with the $(x_1, x_3)-$plane and $\Pi_2$ possesses a dihedral angle of $\alpha\cdot\pi$ away from $\Pi_{1}$ in the anti-clockwise direction. The considering point $\mathbf{ x}_0\in\bsl$ is assumed to be located at the origin $\mathbf{ x}_0=\mathbf{0}$. Similar to \cite{CDL3}, with the help of analytic continuation property for $u$, we can restrict our discussion to $\alpha\in (0,1)$.

In order to utilize the spherical wave expansion method to discuss  the vanishing properties of $u$, we first introduce several propositions and lemmas based on the spherical coordinates in $\R^3$. For any point $\mathbf{ x}\in\R^3$, we denote
\begin{equation}\label{eq-coordi}
\mathbf{ x}=(x_1,x_2,x_3)=(r\sin\theta\cos\phi,r\sin\theta\sin\phi,r\cos\theta):=(r,\theta,\phi),
\end{equation}
where $r\geq0, \theta\in[0,\pi)$ and $\phi\in[0,2\pi)$. It is known that the Laplacian eigenfunction $u$ possesses the spherical wave expansion as follows.

\begin{lemma}\cite[Section 3.3]{CK}\label{expansion}
	The solution $u$ to \eqref{eq:eig} has the spherical wave expansion in spherical coordinates around the origin:
	\begin{equation}\label{expan}
	u(\mathbf{ x})=4\pi\sum_{n=0}^{\infty}\sum_{m=-n}^{n}\bsi^na_n^mj_n(\sqrt{\lambda}r)Y_n^m(\theta, \phi),
	\end{equation}
	where \mm{$j_n(\sqrt{\lambda}r)$ is the spherical Bessel function of order $n$, and} 
	$Y_n^m(\theta,\phi)$ is the spherical harmonics given by 
	\begin{equation*}\label{sphe harmonic}
	Y_n^m(\theta,\phi)=\sqrt{\frac{2n+1}{4\pi}\frac{(n-|m|)!}{(n+|m|)!}}P_n^{|m|}(\cos\theta)e^{\bsi m\phi}
	\end{equation*}
	\mm{with $P_n^m(\cos\theta)$ being the associated Legendre functions.}
\end{lemma}

In particular, the associated Legendre polynomials $P_n^m$ fulfill the following orthogonal relation.

\begin{lemma}\cite[Theorem 2.4.4]{Ned}\label{base2}
	\mm{In the spherical coordinate system, the associated Legendre functions fulfill the following orthogonality condition 
	for any fixed $n \in \mathbb N$, and any two integers $m\ge 0$ and $l\leq n$:}
	\begin{equation*}\label{ortho3}
	\int_{-\pi}^{\pi}\frac{P_n^m(\cos\theta)P_n^l(\cos\theta)}{\sin\theta}\,d\theta=
	\left\{  
	\begin{array}{cc}
	0&\mbox{ if }\quad l\neq m\\
	\frac{(n+m)!}{m(n-m)!}&\mbox{ if }\quad l=m
	\end{array}
	\right. .
	\end{equation*}
\end{lemma}

Furthermore, from \cite{WB}, we know that there holds the recursive relations for $P_n^{m}$ as follows.

\begin{lemma}\label{recur-P}
	In the spherical coordinate system, the associated Legendre functions $P_n^{m}$ fulfills the following recursive equations for any fixed $n, m \in \mathbb N$:
	\begin{equation}\label{recur-p1}
	\frac{dP_n^{|m|}(\cos\theta)}{d\theta}=\frac{1}{2}\left((n+|m|)(n-|m|+1)P_n^{|m|-1}-P_n^{|m|+1}\right),
	\end{equation}
	and 
	\begin{equation}\label{recur-p2}
	\frac{P_n^{|m|}(\cos\theta)}{\sin\theta}=-\frac{1}{2|m|}\left(P_{n+1}^{|m|+1}(\cos\theta)+(n-|m|+1)(n-|m|+2)P_{n+1}^{|m|-1}(\cos\theta)\right),
	\end{equation}
	which can be also represented as 
	\begin{equation}\label{recur-pn}
	\frac{P_n^{|m|}(\cos\theta)}{\sin\theta}=-\frac{1}{2|m|}\left(P_{n-1}^{|m|+1}(\cos\theta)+(n+|m|-1)(n+|m|)P_{n-1}^{|m|-1}(\cos\theta)\right),
	\end{equation}
	with slight changes of the lower index.
\end{lemma}

We proceed to present our main theorem regarding the theoretical study of the vanishing properties of Laplacian eigenfunctions at any edge corner intersected by two generalized singular planes. For simplification in the subsequent study, we assume the impedance boundary parameter $\eta$ to be a nonzero constant. Nevertheless, it is pointed out that the same results hold for any real analytic function $\eta$ with the form \eqref{vari-eta0} fulfilling that  $\alpha_{0}\neq0$, where  $\alpha_{0}$ is the constant part of \eqref{vari-eta0}; see Remark \ref{rem-eta func} for more relevant explanations.

\begin{theorem}\label{main-vani}
	Let $u$ be a Laplacian eigenfunction to \eqref{eq:eig}. Consider an edge corner ${\mathcal E}(\Pi_1, \Pi_2,\bsl)\Subset \Omega$ associated with two generalized singular planes $\Pi_{1}$ and $\Pi_2$. If 
	the corresponding dihedral angle can be written as
	\begin{equation}\notag
	\angle(\Pi_{1}, \Pi_2)=\phi_0=\alpha\cdot\pi,\quad \alpha\in(0,1),
	\end{equation}
	where $\alpha$ satisfies that for any $N\in\mathbb{N}$, $N\geq 2$,  
	\begin{equation}\label{angle}
	\alpha\neq \frac{q}{p}, \ \  p=1,2,\cdots, N-1, \  q=0, 1,2,\cdots, p-1,
	\end{equation}
	and the corresponding impedance parameters $\eta_i$  associated with $\Pi_i$, $i=1,2,$ fulfill that
	\begin{equation}\label{new-eta}
	2\eta_1\cos\phi_0+\eta_2(1+\cos2\phi_0)\neq0,	
	\end{equation}
	then $u$ vanishes up to the order at least $N$ at the edge corner $\mathbf 0$. 
\end{theorem}

\begin{proof}
	Since $\Pi_{1}$ and $\Pi_2$ are two generalized singular planes such that $\Pi_{1}\cap\Pi_2=\bsl$, it is obvious that there holds
	\begin{equation}\label{orgi-pi}
		\frac{\partial u}{\partial\nu_1}+\eta_1u=0 \ \mbox{ on }\Pi_1,\quad \frac{\partial u}{\partial\nu_2}+\eta_2u=0 \ \mbox{ on }\Pi_2,
	\end{equation}
	where $\nu_i$, $i=1,2$, are the unit normal vectors perpendicular to $\Pi_{1}$ and $\Pi_2$, respectively. $\eta_1$ and $\eta_2$ are the corresponding impedance parameters. It is easy to know that $\nu_1=(0, -1, 0)$ and for any fixed $\Pi_2$, we denote $\nu_2=(-\sin\phi_0, \cos\phi_0, 0)$; seeing Figure \ref{fig:main-vani} for the schematic illustration in spherical coordinate system.
	
	\begin{figure}[h!]
		\centering
		\includegraphics[width=0.3\linewidth]{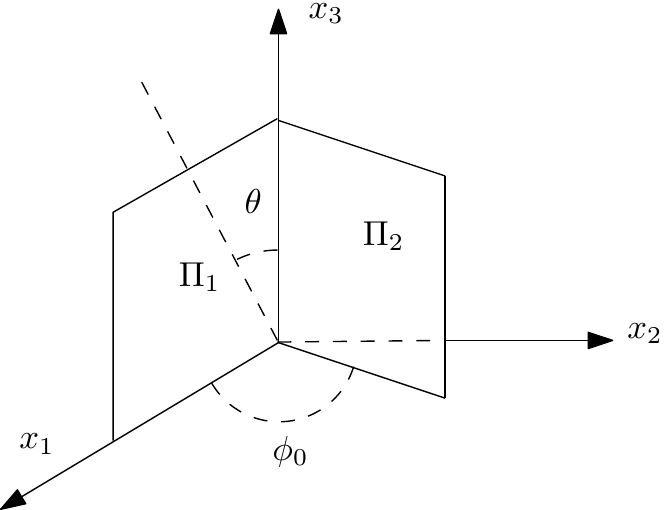}
		\caption{A schematic illustration for an edge corner with the dihedral angle $\phi_0$}
		\label{fig:main-vani}
	\end{figure}
	
	Thus, \eqref{orgi-pi} can be written more explicitly as
	\begin{align}
		\frac{\partial u}{\partial\nu_1}+\eta_1u&=-\frac{\partial u}{\partial x_2}+\eta_1u=0,&\mbox{on }\Pi_{1},\label{pi-21}\\
		\frac{\partial u}{\partial\nu_2}+\eta_2u&=-\frac{\partial u}{\partial x_1}\sin\phi_0+\frac{\partial u}{\partial x_2}\cos\phi_0+\eta_2u=0,&\mbox{on }\Pi_2\label{pi-22}.
	\end{align}
	Substituting \eqref{pi-22} into \eqref{pi-21}, we have that
	\begin{equation}\label{bsl0}
		-\frac{\partial u}{\partial x_1}\sin\phi_0+(\eta_1\cos\phi_0+\eta_2)u=0\quad\mbox{on}\quad \bsl=\Pi_{1}\cap\Pi_2.
	\end{equation}
	By utilizing the spherical coordinate system, there holds
	\begin{align}\label{change}
		\frac{\partial u}{\partial x_1}&=\frac{\partial u}{\partial r}\cdot\frac{\partial r}{\partial x_1}+\frac{\partial u}{\partial\theta}\cdot\frac{\partial \theta}{\partial x_1}+\frac{\partial u}{\partial \phi}\cdot\frac{\partial \phi}{\partial x_1}\notag\\
		&=\frac{\partial u}{\partial r}\sin\theta\cos\phi+\frac{\partial u}{\partial\theta}\cdot\frac{\cos\theta\cos\phi}{r}-\frac{\partial u}{\partial \phi}\cdot\frac{\sin\phi}{r\sin\theta}.
	\end{align}
	Then rewriting \eqref{bsl0} in terms of \eqref{change}, we can know that
	\begin{equation}\label{mid1}
		\left(-\frac{\partial u}{\partial r}\sin\theta\cos\phi-\frac{\partial u}{\partial \theta}\cdot\frac{\cos\theta\cos\phi}{r}+\frac{\partial u}{\partial \phi}\cdot\frac{\sin\phi}{r\sin\theta}\right)\cdot\sin\phi_0+(\eta_1\cos\phi_0+\eta_2)u=0\quad\mbox{on }\bsl.
	\end{equation}
	From \eqref{expan} in Lemma \ref{expansion}, which gives the spherical wave expansion of the Laplacian eigenfunction $u$, we can derive from \eqref{mid1} that
	\begin{align}\label{mid2}
	(&-4\pi\sum_{n=0}^\infty\sum_{m=-n}^n\bsi^n a_n^m \frac{dj_n(\sqrt{\lambda}r)}{dr}\sqrt{\frac{2n+1}{4\pi}}\sqrt{\frac{(n-|m|)!}{(n+|m|)!}}P_n^{|m|}(\cos\theta)e^{\bsi m \phi}\sin\theta\cos\phi\notag\\
	&-4\pi \sum_{n=0}^\infty\sum_{m=-n}^n\bsi^na_n^m j_n(\sqrt{\lambda}r)\sqrt{\frac{2n+1}{4\pi}}\sqrt{\frac{(n-|m|)!}{(n+|m|)!}}\frac{dP_n^{|m|}(\cos\theta)}{d\theta}e^{\bsi m \phi}\frac{\cos\theta\cos\phi}{r}\notag\\
	&+4\pi\sum_{n=0}^\infty\sum_{m=-n}^n\bsi^{n+1}m a_n^mj_n(\sqrt{\lambda}r)\sqrt{\frac{2n+1}{4\pi}}\sqrt{\frac{(n-|m|)!}{(n+|m|)!}}P_n^{|m|}(\cos\theta)e^{\bsi m \phi}\frac{\sin\phi}{r\sin\theta})\sin\phi_0\notag\\
	&+(\eta_1\cos\phi_0+\eta_2)4\pi\sum_{n=0}^\infty\sum_{m=-n}^n\bsi^n a_n^m j_n(\sqrt{\lambda}r)\sqrt{\frac{2n+1}{4\pi}}\sqrt{\frac{(n-|m|)!}{(n+|m|)!}}P_n^{|m|}(\cos\theta)e^{\bsi m \phi}=0.
	\end{align}
 Using the recursive equations \eqref{recur-p1} and \eqref{recur-p2} with respect to $P_n^{|m|}$ in Lemma \ref{recur-P},
	and by taking $\theta=0$ and $\phi=\phi_0$ on $\bsl$ in \eqref{mid2}, we can obtain that
	\begin{align}\label{mid3}
		(&-\sum_{n=0}^\infty\sum_{m=-n}^m \bsi^n a_n^m\frac{j_n{\sqrt{\lambda}}r}{r}\sqrt{\frac{2n+1}{4\pi}}\sqrt{\frac{(n-|m|)!}{(n+|m|)!}}\frac{dP_n^{|m|}(\cos\theta)}{d\theta}\big|_{\theta=0}e^{\bsi m \phi_0}\cos\phi_0\notag\\
		&+\sum_{n=0}^\infty\sum_{m=-n}^n\bsi^{n+1}m a_n^m\frac{j_n{\sqrt{\lambda}}r}{r}\sqrt{\frac{2n+1}{4\pi}}\sqrt{\frac{(n-|m|)!}{(n+|m|)!}}e^{\bsi m \phi_0}\sin\phi_0\frac{P_n^{|m|}(\cos\phi_0)}{\sin\phi_0}\big|_{\theta=0})\sin\phi_0\notag\\
		&+(\eta_1\cos\phi_0+\eta_2)\sum_{n=0}^\infty\sum_{m=-n}^n\bsi^na_n^mj_n(\sqrt{\lambda}r)\sqrt{\frac{2n+1}{4\pi}}\sqrt{\frac{(n-|m|)!}{(n+|m|)!}}P_n^{|m|}(1)e^{\bsi m \phi_0}=0,
	\end{align}
	where 
	\begin{equation}\label{mid-p1}
		\frac{dP_n^{|m|}(\cos\theta)}{d\theta}\big|_{\theta=0}=\frac{1}{2}\left((n+|m|)(n-|m|+1)P_n^{|m|-1}(1)-P_n^{|m|+1}(1)\right),
	\end{equation}
	and
	\begin{equation}\label{mid-p2}
		\frac{P_n^{|m|}(\cos\theta)}{\sin\theta}\big|_{\theta=0}=-\frac{1}{2|m|}\left(P_{n+1}^{|m|+1}(1)+(n-|m|+1)(n-|m|+2)P_{n+1}^{|m|-1}(1)\right).
	\end{equation}
	Since we have from \cite{WB} that
	\begin{equation}\notag
		P_n^{|m|}(\pm1)\equiv0,\mbox{ for }m\neq0, \quad
		P_n^0(1)\equiv1,
	\end{equation}
	\eqref{mid3} can be further simplified as
	\begin{align}\label{mid4}
		(&-\sum_{n=0}^\infty\bsi^n\frac{j_n(\sqrt{\lambda}r)}{r}\sqrt{\frac{2n+1}{4\pi}}(a_n^0\frac{1}{2}n(n+1)P_n^{-1}(1)+(a_n^1e^{\bsi\phi_0}+a_n^{-1}e^{-\bsi\phi_0})\sqrt{\frac{(n-1)!}{(n+1)!}}\frac{1}{2}(n+1)n)\cos\phi_0\notag\\
		&-\sum_{n=0}^\infty\bsi^{n+1}\frac{j_n(\sqrt{\lambda}r)}{r}\sqrt{\frac{2n+1}{4\pi}}\sqrt{\frac{(n-1)!}{(n+1)!}}(a_n^1e^{\bsi\phi_0}-a_n^{-1}e^{-\bsi\phi_0})\sin\phi_0\frac{1}{2}n(n+1))\sin\phi_0\notag\\
		&+(\eta_1\cos\phi_0+\eta_2)\sum_{n=0}^\infty\bsi^na_n^0j_n(\sqrt{\lambda}r)\sqrt{\frac{2n+1}{4\pi}}=0,
	\end{align}
	where 
	\begin{equation}\notag
		P_n^{-1}(1)=(-1)^1\frac{(n-1)!}{(n+1)!}P_n^1(1)=-\frac{1}{n(n+1)}P_n^1(1)=0.
	\end{equation}
	Therefore, we can rewrite \eqref{mid4} as
	\begin{align}\label{mid5}
		&-\sum_{n=0}^\infty\bsi^n\frac{j_n(\sqrt{\lambda}r)}{r}\sqrt{\frac{2n+1}{4\pi}}\frac{1}{2}n(n+1)\sqrt{\frac{(n-1)!}{(n+1)!}}\sin\phi_0(a_n^1e^{\bsi\phi_0}\cdot e^{\bsi\phi_0}+a_n^{-1}e^{-\bsi\phi_0}\cdot e^{-\bsi\phi_0})\notag\\
		&+(\eta_1\cos\phi_0+\eta_2)\sum_{n=0}^\infty\bsi^n a_n^0j_n(\sqrt{\lambda}r)\sqrt{\frac{2n+1}{4\pi}}=0.
	\end{align}
	Since $j_n(\sqrt{\lambda}r)$ fulfills the following recursive equations (\cite{GBA})
	\begin{equation}\label{recur-j}
		\frac{j_n(\sqrt{\lambda}r)}{r}=\frac{\sqrt{\lambda}}{2n+1}\left(j_{n-1}(\sqrt{\lambda}r)+j_{n+1}(\sqrt{\lambda}r)\right),
	\end{equation}
	we have
	\begin{align}\label{bsl}
		&-\sum_{n=1}^\infty\bsi^n\frac{\sqrt{\lambda}}{2n+1}\left(j_{n-1}(\sqrt{\lambda}r)+j_{n+1}(\sqrt{\lambda}r)\right)\sqrt{\frac{2n+1}{4\pi}}\frac{1}{2}n(n+1)\sqrt{\frac{(n-1)!}{(n+1)!}}\sin\phi_0(a_n^1e^{\bsi 2\phi_0}+a_n^{-1}e^{-\bsi 2\phi_0})\notag\\
		&+(\eta_1\cos\phi_0+\eta_2)\sum_{n=0}^\infty\bsi^n a_n^0j_n(\sqrt{\lambda}r)\sqrt{\frac{2n+1}{4\pi}}=0,\quad\mbox{on}\quad\bsl.
	\end{align}
	
	Recall that on $\Pi_{1}$, it holds 
	\begin{equation}\notag
		\frac{\partial u}{\partial \nu_1}+\eta_1 u=0\quad\mbox{on}\quad \Pi_{1}. 
	\end{equation}
	By the spherical wave expansion \eqref{expan}, we know that
	\begin{align}\label{Pi1}
		&-\frac{1}{r\sin\theta}4\pi\sum_{n=0}^\infty\sum_{m=-n}^n \bsi^{n+1}m a_n^m j_n{\sqrt{\lambda}r}\sqrt{\frac{2n+1}{4\pi}}\sqrt{\frac{(n-|m|)!}{(n+|m|)!}}P_n^{|m|}(\cos\theta)\notag\\
		&+\eta_1 4\pi \sum_{n=0}^\infty\sum_{m=-n}^n \bsi^n a_n^m{\sqrt{\lambda}r}\sqrt{\frac{2n+1}{4\pi}}\sqrt{\frac{(n-|m|)!}{(n+|m|)!}}P_n^{|m|}(\cos\theta)=0\quad\mbox{on } \Pi_{1}.
	\end{align}
	Recall the recursive equation \eqref{recur-pn} in Lemma \ref{recur-P} with respect to $\theta$.
	By substituting \eqref{recur-j} and \eqref{recur-pn} into \eqref{Pi1}, we can deduce that on $\Pi_{1}$
	\begin{align}\label{Pi1-2}
		&\sum_{n=1}^\infty\sum_{{m=-n}\atop{m\neq0}}^{n}\bsi^{n+1}m a_n^m\frac{\sqrt{\lambda}}{2n+1}\left(j_{n-1}(\sqrt{\lambda}r)+j_{n+1}(\sqrt{\lambda}r)\right)\sqrt{\frac{2n+1}{4\pi}}\sqrt{\frac{(n-|m|)!}{(n+|m|)!}}\notag\\
		&\times \frac{1}{2|m|}\left(P_{n-1}^{|m|+1}(\cos\theta)+(n+|m|-1)(n+|m|)P_{n-1}^{|m|-1}(\cos\theta)\right)\notag\\
		&+\eta_1\sum_{n=0}^\infty\sum_{m=-n}^n\bsi^n a_n^m j_n(\sqrt{\lambda}r)\sqrt{\frac{2n+1}{4\pi}}\sqrt{\frac{(n-|m|)!}{(n+|m|)!}}P_n^{|m|}(\cos\theta)\quad\mbox{on}\quad\Pi_{1},
	\end{align}
	
	Similarly, on $\Pi_2$  we have
	\begin{equation}\notag
		\frac{\partial u}{\partial \nu_2}+\eta_2u=0\quad\mbox{on}\quad \Pi_2.
	\end{equation}
	It yields from \eqref{expan} that
	\begin{align}\label{Pi2-1}
		&\frac{1}{r\sin\theta}4\pi\sum_{n=0}^\infty\sum_{m=-n}^n\bsi^{n+1}m a_n^mj_n(\sqrt{\lambda}r)\sqrt{\frac{2n+1}{4\pi}}\sqrt{\frac{(n-|m|)!}{(n+|m|)!}}P_n^{|m|}(\cos\theta)e^{\bsi m \phi_0}\notag\\
		&+\eta_24\pi\sum_{n=0}^\infty\sum_{m=-n}^n\bsi^n a_n^m j_n(\sqrt{\lambda}r)\sqrt{\frac{2n+1}{4\pi}}\sqrt{\frac{(n-|m|)!}{(n+|m|)!}}P_n^{|m|}(\cos\theta)e^{\bsi m \phi_0}\quad\mbox{on}\quad \Pi_2.
	\end{align}
	Combining with \eqref{recur-j} and \eqref{recur-pn}, we can obtain from \eqref{Pi2-1} that
	\begin{align}\label{Pi2-2}
		&-\sum_{n=1}^\infty\sum_{{m=-n}\atop{m\neq0}}^n\bsi^{n+1}m a_n^m\frac{\sqrt{\lambda}}{2n+1}\left(j_{n-1}(\sqrt{\lambda}r)+j_{n+1}(\sqrt{\lambda}r)\right)\sqrt{\frac{2n+1}{4\pi}}\sqrt{\frac{(n-|m|)!}{(n+|m|)!}}e^{\bsi m\phi_0}\notag\\
		&\times  \frac{1}{2|m|}\left(P_{n-1}^{|m|+1}(\cos\theta)+(n+|m|-1)(n+|m|)P_{n-1}^{|m|-1}(\cos\theta)\right)\notag\\
		&+\eta_2\sum_{n=0}^\infty\sum_{m=-n}^n\bsi^n a_n^mj_n(\sqrt{\lambda}r)\sqrt{\frac{2n+1}{4\pi}}\sqrt{\frac{(n-|m|)!}{(n+|m|)!}}P_n^{|m|}(\cos\theta)e^{\bsi m\phi_0}=0\quad\mbox{on}\quad\Pi_2.
	\end{align}
	
	Now, we apply mathematical induction by investigating the coefficients with respect to $r^n$ for $n=0,1,2,\cdots$ in  \eqref{bsl} on $\bsl$ ,\eqref{Pi1-2} on $\Pi_{1}$ and \eqref{Pi2-2} on $\Pi_2$, separately, to prove our main results. Precisely speaking, for the coefficients of $r^0$, from \eqref{bsl} we first have
	\begin{equation}\label{r01}
		-\bsi \frac{\sqrt{\lambda}}{3}\sqrt{\frac{3}{4\pi}}\sqrt{\frac{1}{2}}\sin\phi_0(a_1^1e^{\bsi 2\phi_0}+a_1^{-1}e^{-\bsi 2\phi_0})+(\eta_1\cos\phi_0+\eta_2)a_0^0\sqrt{\frac{1}{4\pi}}=0,
	\end{equation} 
	then we can show that the coefficients of $r^0$ in \eqref{Pi1-2} and \eqref{Pi2-2} respectively fulfill
	\begin{equation}\label{r02}
		-(a_1^1-a_1^{-1})\frac{\sqrt{\lambda}}{3}\sqrt{\frac{3}{4\pi}}\sqrt{\frac{1}{2}}P_0^0(\cos\theta)+\eta_1a_0^0\sqrt{\frac{1}{4\pi}}P_0^0(\cos\theta)=0,
	\end{equation}
	and 
	\begin{equation}\label{r03}
		(a_1^1e^{\bsi \phi_0}-a_1^{-1}e^{-\bsi\phi_0})\frac{\sqrt{\lambda}}{3}\sqrt{\frac{3}{4\pi}}\sqrt{\frac{1}{2}}P_0^0(\cos\theta)+\eta_2a_0^0\sqrt{\frac{1}{4\pi}}P_0^0(\cos\theta)=0.
	\end{equation}
	For simplification, we further rewrite \eqref{r01}, \eqref{r02} and \eqref{r03} more briefly as the following system:
	\begin{equation}\label{r04}
		\begin{cases}
		&-\bsi \sqrt{\frac{\lambda}{6}}\sin\phi_0(a_1^1e^{\bsi 2\phi_0}+a_1^{-1}e^{-\bsi 2\phi_0})+(\eta_1\cos\phi_0+\eta_2)a_0^0=0,\\
		&(a_1^1-a_1^{-1})\sqrt{\frac{\lambda}{6}}-\eta_1a_0^0=0,\\
		&(a_1^1e^{\bsi \phi_0}-a_1^{-1}e^{-\bsi\phi_0})\sqrt{\frac{\lambda}{6}}+\eta_2a_0^0=0.
		\end{cases}
	\end{equation}
	It is easy to see that the determinant of the coefficient matrix of \eqref{r04} with respect to $a_1^{\pm1}$ and $a_0$ fulfills
	\begin{align}\notag
	&\left|	\begin{array}{ccc}
	-\bsi\frac{\lambda}{6}\sin\phi_0e^{\bsi 2\phi_0}	& -\bsi\frac{\lambda}{6}\sin\phi_0e^{-\bsi 2\phi_0} & \eta_1\cos\phi_0+\eta_2 \\ 
	\frac{\lambda}{6}	& -\frac{\lambda}{6} & -\eta_1 \\ 
	\frac{\lambda}{6}e^{\bsi\phi_0}	& -\frac{\lambda}{6}e^{-\bsi\phi_0}  & \eta_2 
		\end{array} \right|\notag\\
	&	=\frac{\lambda}{36}\left((\eta_1\cos\phi_0+\eta_2)2\bsi\sin\phi_0+\eta_1\bsi\sin2\phi_0+\eta_22\bsi\sin\phi_0\cos2\phi_0\right)\notag\\
		&=\frac{\bsi\lambda}{18}\sin\phi_0(2\eta_1\cos\phi_0+\eta_2(\cos2\phi_0+1)). 
	\end{align}
	By the condition \eqref{angle}, we can know that  $\phi_0\neq0,   \pi$. Due to \eqref{new-eta},  one knows that $a_1^{\pm1}=a_0^0=0$.

	We proceed to compare the coefficients of $r^1$ on the both sides of \eqref{bsl}, \eqref{Pi1-2} and \eqref{Pi2-2}, and we can derive that
	\begin{equation}\notag
		\frac{\sqrt{\lambda}}{5}\frac{\sqrt{\lambda}}{3}\sqrt{\frac{5}{4\pi}}3\sqrt{\frac{1}{6}}\sin\phi_0(a_2^1e^{\bsi 2\phi_0}+a_2^{-1}e^{-\bsi 2\phi_0})+(\eta_1\cos\phi_0+\eta_2)\bsi a_1^0\frac{\sqrt{\lambda}}{3}\sqrt{\frac{3}{4\pi}}=0\quad\mbox{on} \ \bsl,
	\end{equation}
	\begin{align}\label{r1-1}
		&-\bsi\left(2(a_2^2-a_2^{-2})\sqrt{\frac{1}{4!}}3P_1^1(\cos\theta)+(a_2^1-a_2^{-1})\sqrt{\frac{1}{3!}}3P_1^0(\cos\theta)\right)\frac{\sqrt{\lambda}}{5}\frac{\sqrt{\lambda}}{3}\sqrt{\frac{5}{4\pi}}\notag\\
		&+\eta_1\bsi\frac{\sqrt{\lambda}}{3}\sqrt{\frac{3}{4\pi}}\left((a_1^1+a_1^{-1})\sqrt{\frac{1}{2!}}P_1^1(\cos\theta)+a_1^0P_1^0(\cos\theta)\right)=0\quad\mbox{on } \Pi_{1},
	\end{align}
	and 
	\begin{align}\label{r1-2}
		&\bsi\frac{\sqrt{\lambda}}{5}\frac{\sqrt{\lambda}}{3}\sqrt{\frac{5}{4\pi}}\left(2(a_2^2e^{\bsi 2\phi_0}-a_2^{-2}e^{-\bsi2\phi_0})\sqrt{\frac{1}{4!}}3P_1^1(\cos\theta)+(a_2^1e^{\bsi \phi_0}-a_2^{-1}e^{-\bsi\phi_0})\sqrt{\frac{1}{3!}}3P_1^0(\cos\theta)\right)\notag\\
		&+\eta_2\bsi\frac{\sqrt{\lambda}}{3}\sqrt{\frac{3}{4\pi}}\left((a_1^1e^{\bsi\phi_0}+a_1^{-1}e^{-\bsi\phi_0})\sqrt{\frac{1}{2!}}P_1^1(\cos\theta)+a_1^0P_1^0(\cos\theta)\right)=0\quad\mbox{on }\Pi_2.
	\end{align}
	Substituting $a_1^{\pm1}=a_0^0=0$ into \eqref{r1-1} and \eqref{r1-2}, we have
	\begin{align}\label{r1-3}
		&\sqrt{\frac{3\lambda}{10}}\sin \phi_0(a_2^1e^{\bsi 2\phi_0}+a_2^{-1}e^{-\bsi2\phi_0})+(\eta_1\cos\phi_0+\eta_2)\bsi a_1^0 \sqrt{3}=0,\notag\\
		&-\bsi\left(6(a_2^2-a_2^{-2})\sqrt{\frac{1}{4!}}P_1^1(\cos\theta)+3(a_2^1-a_2^{-1})\sqrt{\frac{1}{3!}}P_1^0(\cos\theta)\right)\sqrt{\frac{\lambda}{5}}+\eta_1\bsi\sqrt{3}a_1^0P_1^0(\cos\theta)=0,\notag\\
		&\bsi\sqrt{\frac{\lambda}{5}}\left(6(a_2^2e^{\bsi2\phi_0}-a_2^{-2}e^{-\bsi2\phi_0})\sqrt{\frac{1}{4!}}P_1^1(\cos\theta)+3(a_2^1e^{\bsi\phi_0}-a_2^{-1}e^{-\bsi\phi_0})\sqrt{\frac{1}{3!}}P_1^0(\cos\theta)\right)\notag\\
		&+\eta_2\bsi\sqrt{3}a_1^0P_1^0(\cos\theta)=0.
	\end{align}
	 After rearranging \eqref{r1-3} in terms of $P_1^1(\cos\theta)$ and $P_1^0(\cos\theta)$, we can know that
	\begin{equation}\label{r1-41}
		\sqrt{\frac{\lambda}{10}}\sin\phi_0(a_2^1e^{\bsi2\phi_0}+a_2^{-1}e^{-\bsi2\phi_0})+(\eta_1\cos\phi_0+\eta_2)\bsi a_1^0=0,
	\end{equation}
	\begin{equation}\label{r1-42}
		-\bsi 6\sqrt{\frac{1}{4!}}\sqrt{\frac{\lambda}{5}}(a_2^2-a_2^{-2})P_1^1(\cos\theta)+\bsi\left(\sqrt{3}\eta_1a_1^0-3(a_2^1-a_2^{-1})\sqrt{\frac{1}{3!}}\sqrt{\frac{\lambda}{5}}\right)P_1^0(\cos\theta)=0,
	\end{equation}
	and 
	\begin{equation}\label{r1-43}
		\bsi6\sqrt{\frac{1}{4!}}\sqrt{\frac{\lambda}{5}}(a_2^2e^{\bsi2\phi_0}-a_2^{-2}e^{-\bsi 2\phi_0})P_1^1(\cos\theta)+\bsi \left(\sqrt{3}\eta_2a_1^0+3(a_2^1e^{\bsi\phi_0}-a_2^{-1}e^{-\bsi\phi_0})\sqrt{\frac{1}{3!}}\sqrt{\frac{\lambda}{5}}\right)P_1^0(\cos\theta)=0.
	\end{equation}
	In \eqref{r1-42} and \eqref{r1-43}, using the orthogonal property of the associated Legendre polynomials in Lemma \ref{base2}, we can obtain that
	\begin{equation}\label{a22}
		\begin{cases}
		a_2^2-a_2^{-2}=0,\\
		\eta_1\sqrt{3}a_1^0-3\sqrt{\frac{1}{3!}}\sqrt{\frac{\lambda}{5}}(a_2^1-a_2^{-1})=0.
		\end{cases}
	\end{equation}
	and 
	\begin{equation}\label{a222}
		\begin{cases}
		a_2^2e^{\bsi2\phi_0}-a_2^{-2}e^{-\bsi2\phi_0}=0,\\
		\eta_2\sqrt{3}a_1^0+3\sqrt{\frac{1}{3!}}\sqrt{\frac{\lambda}{5}}(a_2^1e^{\bsi\phi_0}-a_2^{-1}e^{-\bsi\phi_0})=0.
		\end{cases}
	\end{equation}
	From the first equations in \eqref{a22} and \eqref{a222}, it is easy to see that under condition \eqref{angle}, which indicates $\phi_0\neq0, \frac{\pi}{2} ,\pi$, the determinant of the coefficient matrix of $a_2^{\pm2}$ fulfills that
	\begin{equation}\notag
		\left|
\begin{array}{cc}
1& -1 \\ 
e^{\bsi\phi_0}& e^{-\bsi\phi_0} 
\end{array} 
		\right|
		=2\bsi\sin2\phi_0\neq0,
	\end{equation}
	and therefore we have $a_2^{\pm2}=0$. Now combining with \eqref{r1-41}, \eqref{a22} and \eqref{a222}, we can derive the deteminant of the coefficient matrix with respect to $a_1^0$, $a_2^{\pm1}$ fulfills
	\begin{align}\notag
		&\left|
		\begin{array}{ccc}
		\bsi(\eta_1\cos\phi_0+\eta_2)& \sqrt{\frac{\lambda}{10}} \sin\phi_0e^{\bsi2\phi_0} & \sqrt{\frac{\lambda}{10}} \sin\phi_0e^{-\bsi2\phi_0}  \\ 
		\sqrt{3}\eta_1 & -3\sqrt{\frac{1}{3!}}\sqrt{\frac{\lambda}{5}} & 3\sqrt{\frac{1}{3!}}\sqrt{\frac{\lambda}{5}}  \\ 
		\sqrt{3}\eta_2&3\sqrt{\frac{1}{3!}}\sqrt{\frac{\lambda}{5}}e^{\bsi\phi_0}  & -3\sqrt{\frac{1}{3!}}\sqrt{\frac{\lambda}{5}}e^{-\bsi\phi_0}
		\end{array} 
		\right|\notag\\
		&=\frac{3\lambda}{5}\sin\phi_0(2\eta_1\cos\phi_0+\eta_2+\eta_2\cos2\phi_0)\neq 0.\notag
	\end{align}
	Due  to the fact that $\phi_0\neq0, \pi$, under the condition \eqref{new-eta},   we can obtain $a_1^0=a_2^{\pm1}=0$.
	
	Utilizing mathematical induction, we can assume that $a_{N-2}^0=a_{N-1}^{\pm m}=0$ for $m=1,2, \cdots, N-1$. Then by comparing the coefficients of $r^{N-1}$ on the both sides of \eqref{bsl}, \eqref{Pi1-2} and \eqref{Pi2-2}, it holds: 
	\begin{align}\label{rnl}
		&-\bsi^N\frac{\sqrt{\lambda}}{2N+1}\sqrt{\frac{2N+1}{4\pi}}\frac{1}{2}N(N+1)\sqrt{\frac{(N-1)!}{(N+1)!}}\frac{(\sqrt{\lambda})^{N-1}}{(2N-1)!!}\sin\phi_0(a_N^1e^{\bsi2\phi_0}+a_n^{-1}e^{-\bsi2\phi_0})\notag\\
		&+(\eta_1\cos\phi_0+\eta_2)\bsi^{N-1}a_{N-1}^0\sqrt{\frac{2N-1}{4\pi}}\frac{(\sqrt{\lambda})^{N-1}}{(2N-1)!!}=0\quad\mbox{on}\quad\bsl,
	\end{align}
	\begin{align}\label{rn1}
		&\bsi^{N+1}\frac{\sqrt{\lambda}}{2N+1}\sqrt{\frac{2N+1}{4\pi}}\frac{(\sqrt{\lambda})^{N-1}}{(2N-1)!!}\sum_{{m=-N}\atop{m\neq0}}^Nma_N^m\sqrt{\frac{(N-|m|)!}{(N+|m|)!}}\frac{1}{2|m|}(P_{N-1}^{|m|+1}(\cos\theta)+(N+|m|-1)\notag\\
		&\times (N+|m|)P_{N-1}^{|m|-1}(\cos\theta))+\eta_1\bsi^{N-1}\sum_{m=-(N-1)}^{N-1}a_{N-1}^m\frac{(\sqrt{\lambda})^{N-1}}{(2N-1)!!}\sqrt{\frac{2N-1}{4\pi}}\sqrt{\frac{(N-1-|m|)!}{(N-1+|m|)!}}\notag\\
		&\times  P_{N-1}^{|m|}(\cos\theta)=0\quad\mbox{on}\quad\Pi_1,
	\end{align}
	and
	\begin{align}\label{rn2}
		&-\bsi^{N+1}\frac{\sqrt{\lambda}}{2N+1}\sqrt{\frac{2N+1}{4\pi}}\frac{(\sqrt{\lambda})^{N-1}}{(2N-1)!!}\sum_{{m=-N}\atop{m\neq0}}^Nma_N^me^{\bsi m\phi_0}\sqrt{\frac{(N-|m|)!}{(N+|m|)!}}\frac{1}{2|m|}(P_{N-1}^{|m|+1}(\cos\theta)+(N+|m|-1)\notag\\
		&\times (N+|m|)P_{N-1}^{|m|-1}(\cos\theta))+\eta_2\bsi^{N-1}\sum_{m=-(N-1)}^{N-1}a_{N-1}^m\frac{(\sqrt{\lambda})^{N-1}}{(2N-1)!!}\sqrt{\frac{2N-1}{4\pi}}\sqrt{\frac{(N-1-|m|)!}{(N-1+|m|)!}}\notag\\
		&\times P_{N-1}^{|m|}(\cos\theta)e^{\bsi m\phi_0}=0 \quad\mbox{on}\quad \Pi_2.
	\end{align}
	In \eqref{rn1} and \eqref{rn2}, by using the fact that
	\begin{equation}\notag
		P_n^m=0\quad\mbox{for}\quad |m|>n,
	\end{equation}
	we can deduce from \eqref{rn1} that
	\begin{align}\label{rn1-2}
		&-\sqrt{\frac{\lambda}{2N+1}}[N(a_N^N-a_N^{-N})\sqrt{\frac{1}{(2N)!}}(2N-1)P_{N-1}^{N-1}(\cos\theta)+(N-1)(a_{N}^{N-1}-a_N^{-(N-1)})\notag\\
		&\times \sqrt{\frac{1}{(2N-1)!}}(2N-1)P_{N-1}^{N-2}(\cos\theta)+\sum_{{m=-(N-2)}\atop{m\neq0}}^{N-2}ma_N^m\sqrt{\frac{(N-|m|)!}{(N+|m|)!}}\frac{1}{2|m|}(P_{N-1}^{|m|+1}(\cos\theta)\notag\\
		&+(N+|m|-1)(N+|m|)P_{N-1}^{|m|-1}(\cos\theta))]+\eta_1\sqrt{2N-1}\sum_{m=-(N-1)}^{N-1}a_{N-1}^m\sqrt{\frac{(N-1-|m|)!}{(N-1+|m|)!}}P_{N-1}^{|m|}(\cos\theta)=0.
	\end{align}
	Rearranging terms in \eqref{rn1-2} by virtue of $P_{N-1}^{|m|}$, we have
	\begin{align}\label{rn-mid}
		&\big[-\sqrt{\frac{\lambda}{2N+1}}\left(\right.N(a_N^N-a_N^{-N})\sqrt{\frac{1}{(2N)!}}(2N-1)+(N-2)(a_N^{N-2}-a_N^{-(N-2)})\sqrt{\frac{2}{(2N-2)!}}\frac{1}{2(N-2)}\left.\right)\notag\\
		&+\eta_1\sqrt{2N-1}(a_{N-1}^{N-1}+a_{N-1}^{-({N-1})})\sqrt{\frac{1}{2(N-1)!}}\big]P_{N-1}^{N-1}(\cos\theta)\notag\\
		&+\big[-\sqrt{\frac{\lambda}{2N+1}}\left(\right.(N-1)(a_N^{N-1}-a_N^{-(N-1)})\sqrt{\frac{1}{(2N-1)!}}(2N-1)+(N-3)(a_N^{N-3}-a_N^{-(N-3)})\notag\\
		&\times \sqrt{\frac{3!}{(2N-3)!}}\frac{1}{2(N-3)}\left.\right)+\eta_1\sqrt{2N-1}(a_{N-1}^{N-2}+a_{N-1}^{-({N-2})})\sqrt{\frac{1}{2(N-3)!}}\big]P_{N-1}^{N-2}(\cos\theta)\notag\\
		&+\sum_{m=2}^{N-3}\big[-\sqrt{\frac{\lambda}{2N+1}}\left(\right.(m-1)(a_N^{m-1}-a_N^{-(m-1)})\sqrt{\frac{(N-m+1)!}{(N+m-1)!}}\frac{1}{2(m-1)}+(m+1)(a_N^{m+1}-a_N^{-(m+1)})\notag\\
		&\times \sqrt{\frac{(N-m+1)!}{(N+m-1)!}}\frac{1}{2(m+1)}(N+m)(N+m+1)\left.\right)+\eta_1\sqrt{2N-1}(a_{N-1}^m+a_{N-1}^{-m})\sqrt{\frac{(N-1-m)!}{(N-1+m)!}}\big]P_{N-1}^m(\cos\theta)\notag\\
		&+\big[-\sqrt{\frac{\lambda}{2N+1}}2(a_N^2-a_N^{-2})\sqrt{\frac{(N-2)!}{(N+2)!}}\frac{1}{4}(N+1)(N+2)+\eta_1\sqrt{2N-1}(a_{N-1}^1+a_{N-1}^{-1})\sqrt{\frac{(N-2)!}{N!}}\big]\notag\\
		&\times  P_{N-1}^1(\cos\theta)+\big[-\sqrt{\frac{\lambda}{2N+1}}(a_N^1-a_N^{-1})\sqrt{\frac{(N-1)!}{(N+1)}}\frac{1}{2}N(N+1)+\eta_1\sqrt{2N-1}a_{N-1}^0\big]P_{N-1}^0(\cos\theta)=0.
	\end{align}
	With the help of the orthogonal property of the associated Legendre polynomials in Lemma \ref{base2}, we can further derive from \eqref{rn-mid} that
	\begin{align}\label{rn-mid2}
	&-\sqrt{\frac{\lambda}{2N+1}}\big[N\sqrt{\frac{1}{(2N)!}}(2N-1)(a_N^N-a_N^{-N})+(N-2)\sqrt{\frac{2}{(2N-2)!}}\frac{1}{2(N-2)}(a_N^{N-2}-a_N^{-(N-2)})\big]\notag\\
	&+\eta_1\sqrt{2N-1}\sqrt{\frac{1}{2(N-1)!}}(a_{N-1}^{N-1}+a_{N-1}^{-(N-1)})=0,
	\end{align}
	and
	\begin{align}\label{rn-mid3}
	&-\sqrt{\frac{\lambda}{2N+1}}\big[(N-1)\sqrt{\frac{1}{(2N-1)!}}(2N-1)(a_N^{N-1}-a_N^{-(N-1)})+(N-3)\sqrt{\frac{3!}{(2N-3)!}}\frac{1}{2(N-3)}\notag\\
	&(a_N^{N-3}-a_N^{-(N-3)})\big]
	+\eta_1\sqrt{2N-1}\sqrt{\frac{1}{(2N-3)!}}(a_{N-1}^{N-2}+a_{N-1}^{-(N-2)})=0.
	\end{align}
	Moreover, for $m=2,3,\cdots, N-3$, we have
	\begin{align}\label{rn-mid4}
		&-\sqrt{\frac{\lambda}{2N+1}}\frac{1}{2}\big[\sqrt{\frac{(N-m+1)!}{(N+m-1)!}}(a_N^{m-1}-a_N^{-(m-1)})+\sqrt{\frac{(N-m-1)!}{(N+m+1)!}}(N+m)(N+m+1)\notag\\
		&\times (a_N^{m+1}-a_N^{-(m+1)})\big]+\eta_1\sqrt{2N-1}\sqrt{\frac{(N-1-m)!}{(N-1+m)!}}(a_{N-1}^m+a_{N-1}^{-m})=0,
	\end{align}
	and 
	\begin{equation}\label{rn-mid5}
		-\sqrt{\frac{\lambda}{2N+1}}\frac{1}{2}\sqrt{\frac{(N-2)!}{(N+2)!}}(N+1)(N+2)(a_N^2-a_N^{-2})+\eta_1\sqrt{2N-1}\sqrt{\frac{(N-2)!}{N!}}(a_{N-1}^1+a_{N-1}^{-1})=0,
	\end{equation}
	\begin{equation}\label{rn-mid6}
		-\sqrt{\frac{\lambda}{2N+1}}\frac{1}{2}\sqrt{\frac{(N-1)!}{(N+1)!}}(N+1)N(a_N^1-a_N^{-1})+\eta_1\sqrt{2N-1}a_{N-1}^0=0.
	\end{equation}
	
	Since by induction, there holds $a_{N-1}^m=0$, for $m=1,2,\cdots, N-1$, and $a_{N-2}^0=0$, we can derive from \eqref{rn-mid2}, \eqref{rn-mid3}, \eqref{rn-mid4}, \eqref{rn-mid5} and \eqref{rn-mid6} the following system
	\begin{align}\label{rn-sys}
		&N(2N-1)\sqrt{\frac{1}{(2N)!}}(a_N^N-a_N^{-N})+\frac{1}{2}\sqrt{\frac{2}{(2N-2)!}}(a_N^{N-2}-a_N^{-(N-2)})=0,\notag\\
		&(N-1)(2N-1)\sqrt{\frac{1}{(2N-1)!}}(a_N^{N-1}-a_N^{-(N-1)})+\frac{1}{2}\sqrt{\frac{3}{(2N-3)!}}(a_N^{N-3}-a_N^{-(N-3)})=0,\notag\\
		&\sqrt{\frac{(N-m+1)!}{(N+m-1)!}}(a_N^{m-1}-a_N^{-(m-1)})+\sqrt{\frac{(N-m-1)!}{(N+m+1)!}}(N+m)(N+m+1)(a_N^{m+1}-a_N^{-(m+1)})=0,\notag\\
		&a_N^2-a_N^{-2}=0,\notag\\
		&-\sqrt{\frac{\lambda}{2N+1}}\frac{1}{2}\sqrt{\frac{(N-1)!}{(N+1)!}}(N+1)N(a_N^1-a_N^{-1})+\eta_1\sqrt{2N-1}a_{N-1}^0=0,
	\end{align}
	where $m=2,3, \cdots, N-3$.
	
	Following the similar arguments above of \eqref{rn1} on $\Pi_{1}$, from \eqref{rn2}, we can deduce that on $\Pi_2$, it holds that
	\begin{align}\label{rn-sys2}
		&N(2N-1)\sqrt{\frac{1}{(2N)!}}(a_N^Ne^{\bsi N\phi_0}-a_N^{-N}e^{-\bsi N\phi_0})+\frac{1}{2}\sqrt{\frac{2}{(2N-2)!}}(a_N^{N-2}e^{\bsi(N-2)\phi_0}-a_N^{-(N-2)e^{-\bsi(N-2)\phi_0}})=0,\notag\\
		&(N-1)(2N-1)\sqrt{\frac{1}{(2N-1)!}}(a_N^{N-1}e^{\bsi (N-1)\phi_0}-a_N^{-(N-1)}e^{-\bsi (N-1)\phi_0})+\frac{1}{2}\sqrt{\frac{3!}{(2N-3)!}}\notag\\
		&\cdot(a_N^{N-3}e^{\bsi(N-3)\phi_0}-a_N^{-(N-3)e^{-\bsi(N-3)\phi_0}})=0,\notag\\
		&(a_N^{m-1}e^{\bsi(m-1)\phi_0}-a_N^{-(m-1)}e^{-\bsi(m-1)\phi_0})\sqrt{\frac{(N-m+1)!}{(N+m-1)!}}+(N+m)(N+m+1)\sqrt{\frac{(N-m-1)!}{(N+m-1)!}}\notag\\
		&\cdot(a_N^{m+1}e^{\bsi(m+1)\phi_0}-a_N^{-(m+1)}e^{-\bsi(m+1)\phi_0})=0,\quad\mbox{for }m=2,3,\cdots, N-3,\notag\\
		&a_N^2e^{\bsi2\phi_0}-a_N^{-2}e^{-\bsi2\phi_0}=0,\notag\\
		&\sqrt{\frac{\lambda}{2N+1}}\frac{1}{2}N(N+1)\sqrt{\frac{(N-1)!}{(N+1)!}}(a_N^1e^{\bsi\phi_0}-a_N^{-1}e^{-\bsi\phi_0})+\eta_2\sqrt{2N-1}a_{N-1}^0=0.
	\end{align}
	
	Combining with \eqref{rnl}, \eqref{rn-sys} and \eqref{rn-sys2}, we can directly see that since $\phi_0\neq0, \frac{\pi}{2}, \pi$, the determinant of the coefficient matrix with respect to $a_N^{\pm2}$ satisfies 
	\begin{equation}\notag
		\left|
		\begin{array}{cc}
		1& -1 \\ 
		e^{\bsi2\phi_0}& e^{-\bsi2\phi_0}
		\end{array} 
		\right|=2\bsi\sin2\phi_0\neq0,
	\end{equation}
	which indicates that $a_N^{\pm2}=0$. Moreover, denoting
	\begin{equation}\notag
		c(N,\lambda):=\sqrt{\frac{\lambda}{2N+1}}\frac{N}{2}(N+1)\sqrt{\frac{(N-1)!}{(N+1)!}},
	\end{equation}
	the determinant of the coefficient matrix associated with $a_N^{\pm1}$, and $a_{N-1}^0$ fulfills 
	\begin{align}\notag
		&\left|
		\begin{array}{ccc}
		-\bsi c(N, \lambda)\sin\phi_0e^{\bsi2\phi_0} &-\bsi c(N, \lambda)\sin\phi_0e^{-\bsi2\phi_0}  &(\eta_1\cos\phi_0+\eta_2)\sqrt{2N-1}  \\ 
		-c(N,\lambda)& c(N,\lambda)  & \eta_1\sqrt{2N-1} \\
		c(N,\lambda) e^{\bsi\phi_0}& -c(N,\lambda)e^{-\bsi\phi_0} & \eta_2\sqrt{2N-1} 
		\end{array} 
		\right|\notag\\
		&=-2\bsi\sin\phi_0\sqrt{2N-1}\frac{\lambda}{2N+1}\frac{N^2}{4}(N+1)^2\frac{(N-1)!}{(N+1)!}\left(2\eta_1\cos\phi_0+\eta_2(1+\cos2\phi_0)\right).
	\end{align}
	Therefore, under the condition \eqref{new-eta},  in view of $\phi_0\neq0,  \pi$,  we can obtain that $a_N^{\pm1}=a_{N-1}^0=0$.
	
	Now, by substituting $a_{N-1}^0=a_N^{\pm1}=a_N^{\pm2}=0$ into \eqref{rn-sys} and \eqref{rn-sys2}, respectively, we can deduce that for $m=2$, there holds
	\begin{equation}\notag
		\begin{cases}
		&a_N^3-a_N^{-3}=0,\\
		&a_N^3e^{\bsi3\phi_0}-a_N^{-3}e^{-\bsi3\phi_0}=0.
		\end{cases}
	\end{equation}
	Since $\phi_0\neq\frac{q\pi}{3}$ ($q=0, 1, 2, 3$) under condition \eqref{angle}, we have $a_N^{\pm3}=0$. Following the similar argument, by taking $a_N^{\pm3}=0$ into \eqref{rn-sys} and \eqref{rn-sys2}, we can generally derive that for $m=2,3, \cdots, N$, since $\phi_0\neq  \frac{q\pi}{m}$ ($q=0, 1, 2, \cdots, m-1$), there holds $a_N^{\pm m}=0$ for $m=1,2,\cdots, N$. 
	
	The proof is complete.	
\end{proof}

\begin{remark}\label{rem27}
If the impedance parameters $\eta_i$ associated with $\Pi_i$ ($i=1,2$) fulfill
\begin{equation}\notag
	\eta_1=\eta_2=\eta\neq 0,
\end{equation} 
where $\eta \in \mathbb C$ is a constant, then the condition \eqref{new-eta} can be directly satisfied when $\phi_0\neq 0, \pi/2$,   and $\pi$. Indeed, we have
	\begin{equation}\label{eq:eta cond}
			2\eta_1\cos\phi_0+\eta_2(\cos2\phi_0+1)=2\eta \cos\phi_0 (1+\cos\phi_0) \neq0.
	\end{equation} 	  
\end{remark}

%
%

\begin{remark}\label{rem-eta func}
	The impedance parameter $\eta$ is assume to be a nonzero constant in Theorem \ref{main-vani}. Indeed, the same results hold for more general case that $\eta$ can be a real-analytic function with the series representation under spherical coordinate system
	\begin{equation}\label{eta-func}
		\eta(r,\theta,\phi)=\sum_{\ell=0}^{\infty}\alpha_{\ell}(\theta, \phi)r^{\ell}.
	\end{equation}
	Define $\mathrm{deg}_\eta$ to be the degree of the analytic function $\eta$. More precisely, if $\mathrm{deg}_\eta=N$, it holds that $a_i=0$ for $i=0,1,\cdots, N-1$, and $a_{N}\neq0$. Recall the recursive equations \eqref{bsl}, \eqref{Pi1-2} and \eqref{Pi2-2}, by assuming
	\begin{equation}\label{ad1}
	\eta_1=\sum_{\ell=0}^\infty\alpha_{\ell}^{(1)}(\theta, \phi)r^\ell\quad\mbox{and}\quad \eta_2=\sum_{\ell=0}^\infty\alpha_{\ell}^{(2)}(\theta, \phi)r^\ell,
	\end{equation}
	and substituting \eqref{ad1} into \eqref{bsl}, we can obtain that
	\begin{align}\label{bsl-eta}
	&-\sum_{n=1}^\infty\bsi^n\frac{\sqrt{\lambda}}{2n+1}\left(j_{n-1}(\sqrt{\lambda}r)+j_{n+1}(\sqrt{\lambda}r)\right)\sqrt{\frac{2n+1}{4\pi}}\frac{1}{2}n(n+1)\sqrt{\frac{(n-1)!}{(n+1)!}}\sin\phi_0(a_n^1e^{\bsi 2\phi_0}+a_n^{-1}e^{-\bsi 2\phi_0})\notag\\
	&+\left((\alpha_{0}^{(1)}+\sum_{\ell=1}^\infty\alpha_{\ell}^{(1)}(\theta, \phi)r^\ell)\cos\phi_0+(\alpha_{0}^{(2)}+\sum_{\ell=1}^\infty\alpha_{\ell}^{(2)}(\theta, \phi)r^\ell)\right)\sum_{n=0}^\infty\bsi^n a_n^0j_n(\sqrt{\lambda}r)\sqrt{\frac{2n+1}{4\pi}}=0.
	\end{align}
	It is easy to verify that $\mathrm{deg}_{\eta u}\geq N$, which is dominated by the order of the term $j_{n-1}$. And therefore, the vanishing  property of the Laplacian eigenfunction at the intersecting generalized singular planes with real analytic impedance parameters can be established similarly as Theorem \ref{main-vani}, where the condition \eqref{new-eta} is replaced by
	$$
	2\alpha_{0}^{(1)}\cos\phi_0+\alpha_{0}^{(2)}(\cos2\phi_0+1)\neq 0. 
	$$
\end{remark}

\begin{remark}
	Compared with the study on the geometric structure of Laplacian eigenfunctions in \cite[Theorem 2.11]{CDL3}, it is direct to see that our current results are more general by relaxing the technical condition $u|_{B_\epsilon(\mathbf0)\cap \bsl}\equiv0$, which is relatively hard to be fulfilled in the study on the application of inverse problems for an edge corner. Moreover, the condition \eqref{new-eta} implies that the rationality  on the dihedral angle of two intersecting adjacent planes is sufficient to determine the vanishing orders of Laplacian eigenfunctions. Indeed,  we  need certain assumption on the roots of  the associated Legendre polynomials  to study the  vanishing order of  the underlying Laplacian eigenfunction at a vertex corner; see \cite[Theorem 3.1]{CDL3}  for more details. 
	
\end{remark}


\section{Unique identifiability for inverse problems}\label{sec3}

In this section, we shall establish the unique identifiability results for the inverse scattering problems by using the geometric results established in the previous section. For simplicity, we shall only consider the boundary impedance parameter $\eta$ to be a nonzero constant throughout this section. It is remarked that if $\eta$ is a real-analytic function with the form \eqref{eta-func}, similar unique identifiability results can be obtained.


\subsection{Unique recovery for the inverse obstacle problem}

Recalling the mathematical setup for the inverse obstacle problem in Subsection \ref{subsec-obstacle}, we are going to present the proofs of the uniqueness results for the inverse problem \eqref{inverse} for certain admissible complex polyhedral obstacles defined by Definitions \ref{ad obstacle0} and \ref{def60}.
%


\begin{proof}[Proof of Theorem \ref{inverse10}(Irrational case)]
	    We prove the theorem by absurdity. Assume that
	    there exists a corner
	    $\mathbf x_c$ on $\partial \mathbf{G}$, which is either located at $\Omega_1$ or  $\Omega_2$. Without loss of generality, we assume that $\mathbf x_c$ is a corner of $\Omega_2$, i.e. $\mathbf{ x}_c\in\Omega_2\backslash\overline{ \Omega }_1$. Suppose that $B_h(\mathbf{ x}_c)$ is an open ball centered at $\mathbf{ x}_c$ with sufficiently small $h\in\mathbb{R}_+$ fulfilling that $B_h(\mathbf x_c)\Subset\mathbb{R}^3\backslash\overline \Omega_1 $. Suppose that 
	    \begin{equation}\notag
	    	B_h(\mathbf{ x}_c)\cap \partial\Omega_2=\Pi_i, \quad i=1,2,\cdots, n,
	    \end{equation}
	    for $n\geq2$.

		Recall that $\mathbf{G}$ is the unbounded connected component of $\mathbb{R}^3\backslash\overline{(\Omega_1\cup\Omega_2)}$. From \eqref{eq:cond10}, by the Rellich theorem (cf. \cite{CK}), we have
		\begin{equation}\label{eq:aa3}
		u^1(\mathbf x; k, \mathbf{d})={u}^2(\mathbf x; k, \mathbf{d}),\quad \mathbf x\in\mathbf{G}.
		\end{equation}
		Since $\Pi_i\subset\partial\mathbf{G}$, combining with \eqref{eq:aa3} and the generalized singular boundary condition defined on $\partial\Omega_2$, it is easy to know that
		\begin{equation}\label{eq:aa4}
		\partial_\nu u^1+\eta_2 u^1=\partial_\nu  u^2+\eta_2 u^2=0\quad\mbox{on}\ \ \Pi_i, \ i=1,2,\cdots, n.
		\end{equation}
		Moreover, since $B_h(\mathbf x_c)\Subset\mathbb{R}^3\backslash\overline \Omega_1 $, there holds $-\Delta u^1=k^2 u^1$ in $B_h(\mathbf x_c)$. Recall that the impedance parameter on each faces of $\Omega_2$  is a fixed nonzero constant.  It is obvious to see that for the admissible polyhedral obstacles $\Omega_1$ and $\Omega_2$, the condition \eqref{new-eta} is automatically satisfied by Remark \ref{rem27}. By utilizing Theorem \ref{main-vani} for the Laplacian eigenfunction $u^1$, since $\mathbf{ x}_c$ is an irrational angle which fulfills condition \eqref{angle}, we can obtain that 
		\begin{equation}\label{inver1-ome}
		u^1(\mathbf{ x}; k, \mathbf{d})\equiv0 \mbox{ in } B_h(\mathbf{ x}_c),
		\end{equation}
	which in turn yields by the analytic continuation that 
	\begin{equation}\label{eq:aa51}
	u^1(\mathbf x; k, \mathbf{d})=0\quad\mbox{in}\ \ \mathbb{R}^3\backslash\overline{\Omega}. 
	\end{equation}
	In particular, one has from \eqref{eq:aa51} that
	\begin{equation}\label{eq:aa6}
	\lim_{|\mathbf x|\rightarrow\infty} \left|u^1(\mathbf x; k, \mathbf{d})\right|=0. 
	\end{equation}
	However, we know from the formulations of the forward scattering problem that
	\begin{equation}\label{eq:aa61}
	\lim_{|\mathbf x|\rightarrow\infty} \left|u^1(\mathbf  x; k, \mathbf{d})\right|=\lim_{|\mathbf x|\rightarrow\infty} \left|e^{\mathrm{i}k\mathbf x\cdot \mathbf{d}}+u^s(\mathbf x; k, \mathbf{d})\right|=1, 
	\end{equation}
	which induce the contradiction to \eqref{eq:aa6}.

	Now, we show that $\eta_1=\eta_2$.
	Suppose that $\Gamma\subset \partial\Omega_1\cap\partial\Omega_2$ is an open subset such that $\eta_1\neq \eta_2$ on $\Gamma$. 
	From \eqref{eq:aa3}, we have already known that $u^1=u^2$ in $\mathbb{R}^3\backslash\overline{(\Omega_1\cup\Omega_2)}$, from which we can directly derive that
	\begin{equation}\label{eq:bb6}
	\partial_\nu u^1+\eta_1 u^1=0, \ \ \partial_\nu u^2+\eta_2 u^2=0,\ \ u^1=u^2, \ \ \partial_\nu u^1=\partial_\nu u^2\quad\mbox{on}\ \ \Gamma. 
	\end{equation}
	After rearranging terms in \eqref{eq:bb6}, we have 
	\begin{equation}\label{obst-eta}
	(\eta_1-\eta_2)u^1=0 \quad\mbox{on } \ \Gamma.
	\end{equation}
	Since $\eta_1\neq{\eta_2}$ on $\Gamma$, 
	\mm{we can deduce by direct computing that}
	\begin{equation*}\label{eq:bb7}
	u^1=\partial_\nu u^1=0\quad\mbox{on}\ \ \Gamma.
	\end{equation*}
	By the classcial Holmgren's uniqueness result (cf. \cite{Liu-Zou}), it is easy to obtain that $u^1=0$ in $\mathbb{R}^3\backslash\Omega$. Therefore, we 
	derive the same contradiction as in \eqref{eq:aa6}, which leads to the conslusion.
\end{proof}

In the following, we shall give the detailed proof of Theorem \ref{inverse2} regarding the unique determination for an  admissible complex rational obstacle by a single far-field measurement.


\begin{proof}[Proof of Theorem \ref{inverse2}(Rational case)]
	We prove the theorem by contradiction. Assume that
	there exists a corner
	$\mathbf x_c$ on $\partial \mathbf{G}$. Without loss of generality, we still assume that $\mathbf x_c$ is a corner of $\Omega_2$, i.e. $\mathbf{ x}_c\in\Omega_2\backslash\overline{ \Omega }_1$. Following the same notation of $B_h(\mathbf{ x}_c)$ in Theorem \ref{inverse10} such that $B_h(\mathbf{ x}_c)\cap \partial\Omega_2=\Pi_i$, $i=1,2,\cdots, n$, for $n\geq2$. With the help of the condition \eqref{eq:cond10} and the Rellich lemma, it is direct to verify that \eqref{eq:aa3} and \eqref{eq:aa4} still hold. Moreover, we can know by Theorem \ref{main-vani} that $u^1$ fulfills \eqref{inver1-ome} and
	\begin{equation}\label{contr}
		u^1(\mathbf{ x}_c)=0,\quad\nabla u^1(\mathbf{ x}_c)\neq0.
	\end{equation}
	under the condition \eqref{cond51}.
	However, for the admissible complex rational obstacle $\Omega_2$, under the assumption \eqref{eq:deg cond}, 
	we can know that $\mathbf{ x}_c$ is either an irrational corner or a rational corner of degree $p\geq 3$. In either of the above case, by Theorem \ref{main-vani}  we can obtain that $u^1$ vanishes at least to the second order, which implies that there holds $\nabla u^1(\mathbf{ x}_c)=0$, and this contradicts to  \eqref{contr}. Similar to the proof of Theorem \ref{inverse10}, if we further assume that $\eta_1\neq\eta_2$, then the uniqueness for the impedance parameter $\eta$ can be deduced immediately by the Holmgren's uniqueness principle.
\end{proof}

\begin{remark}
	The uniqueness results with respect to the admissible impedance obstacles and the corresponding argument in Theorem \ref{inverse10} and Theorem \ref{inverse2} are ``localized" in the neighborhood of $B_h(\mathbf{ x}_c)$ based on the generalized Holmgren's principle obtained from Section \ref{sec2}. Therefore, the results are also applicable to other different types of wave incidences such as the point source.
\end{remark}

\begin{remark}
	In Theorem \ref{inverse10} and Theorem \ref{inverse2}, if the considering adimissible polyhedral obstacles are convex, then we can achieve the global unique identifiaility results which indicate that $\Omega_1=\Omega_2$ and also $\eta_1=\eta_2$ simultaneously by a single far-field measurement. The detailed proof of Corollary \ref{coro-inver} is omitted. 
\end{remark}

\begin{remark}
	We would like to point out that the technical condition $\mathcal{L}(\nabla u^j)(\mathbf{ x}_c)\neq0$ in \eqref{cond51} can be satisfied under some generic conditions on $\Omega$. For example, if the diameter of the obstacle $\Omega$ is relatively small compared with the wavelength in the certain regime that $k\cdot\mathrm{diam}(\Omega)\ll1$, then \eqref{cond51} can hold.
\end{remark}

\subsection{Unique recovery for the inverse diffraction grating problem}

In this subsection, we consider the inverse diffraction grating problem in determining a diffraction grating profile as well as its surface parameter in $\R^3$ by a single far-field pattern. 

\begin{lemma}\cite[Lemma 8.1]{CDL2}\label{ortho}
	Let ${\bf{\xi_\ell}}\in \R^3$, $\ell=1,2,\cdots, n$, be $n$ vectors which are distinct from each other. Let $U\subset\R^3$ be any open subset. Then all the functions in the following set are linearly independent:
	\begin{equation}\notag
		\{e^{\bsi{\bf{\xi_\ell}}\cdot\mathbf{ x}}; \mathbf{ x}\in U, \ell=1,2,\cdots, n\}.
	\end{equation} 
\end{lemma}

\begin{proof}[Proof of Theorem \ref{thm-grat0}(Irrational case)]
	We prove this theorem by contradiction. Without loss of generality, we assume that there exists a corner point $\mathbf{ x}_c$ of $\Lambda_f$ lies on $\partial\mathbf{G}\backslash\Lambda_g$. By the wellposedness of the diffraction grating problem \eqref{model0} and the unique continuation property, we can know from \eqref{equi-cond0} that
	\begin{equation}\notag
		u_f(\mathbf{ x}; k, \mathbf{d})=u_g(\mathbf{ x};k, \mathbf{d})\quad\mbox{for}\quad \mathbf{ x}\in\mathbf{G}.
	\end{equation}
	Indeed, define
	\begin{equation}\notag
	v(\mathbf{ x}; k, \mathbf{d}):=u_f(\mathbf{ x}; k, \mathbf{d})-u_g(\mathbf{ x}; k, \mathbf{d}).
	\end{equation}
	Denote $\mathbf{\Sigma}:=\mathbf{G}\cap\{\mathbf{ x}\in\R^3; \mathbf{ x}'\in\R^2, x_3>b\}\subset\R^3$, then it is obvious that $v(\mathbf{ x}; k, \mathbf{d})$ fulfills
	\begin{equation}\notag
		\Delta v+ k^2 v=0\quad\mbox{in} \ \mathbf{\Sigma}; \quad v=0 \ \mbox{on} \ \Gamma_b, 
	\end{equation}
	and the Rayleigh series expansion \eqref{rayley0}, where $\Gamma_b$ is the boundary of $\mathbf{\Sigma}$. Thus, from the uniqueness of the diffraction  grating scattering problem \eqref{model0}, we can know that $v=0$ in $\mathbf{\Sigma}$. Since $u_f(\mathbf{ x}; k,\mathbf{d})$ and $u_g(\mathbf{ x}; k, \mathbf{d})$ are analytic in $\mathbf{G}$, it is direct to derive that $v(\mathbf{ x}; k, \mathbf{d})$ is analytic in $\mathbf{G}$, which implies $v=0$ in $\mathbf{G}$. Therefore, we have $u_f(\mathbf{ x}; k, \mathbf{d})=u_g(\mathbf{ x}; k, \mathbf{d})$ in $\mathbf{G}$.
	
	Since $\mathbf{ x}_c\in\Lambda_f$ lying on $\partial\mathbf{G}\backslash\Lambda_g$, for suffictiently small $h\in\R^+$, suppose $B_h(\mathbf{ x}_c)\Subset\Omega_g$ such that for $n\geq2$
	\begin{equation}\label{B-grat}
		B_h(\mathbf{ x}_c)\cap\Lambda_f=\Pi_i,\quad i=1,2,\cdots, n.
	\end{equation}
	It is clear that $\Pi_i\subset\Lambda_f\backslash\Lambda_g\subset\partial\mathbf{G}$, $i=1,2,\cdots, n$. Following a similar argument in the proof of Theorem \ref{inverse10}, we can obtain that
	\begin{equation}\notag
		u_g(\mathbf{ x}; k,\mathbf{d})=0 \ \mbox{ for } \ x_3>\max_{\mathbf{ x}'\in[0,2\pi)^2} g(\mathbf{ x}'),
	\end{equation}  
	by utilizing the generalized Holmgren's principle and Theorem \ref{main-vani}. Moreover, we know that $u_g(\mathbf{ x}; k, \mathbf{d})$ satisfies the Rayleigh series expansion as follows
	\begin{equation}\label{expan2}
		u_g(\mathbf{ x}; k, \mathbf{d})=e^{\bsi k \mathbf{ d}\cdot\mathbf{x}}+\sum_{\mathbf{n}\in\mathbb{Z}^2}u_n e^{\bsi{\bf \xi_n}\cdot\mathbf{ x}} \ \mbox{ for } \ x_3>\max_{\mathbf{ x}'\in[0,2\pi)^2}g(\mathbf{ x}'),
	\end{equation}
	where ${\bf \xi_n}$ is defined in \eqref{nota0}.
	
	Combining with \eqref{def-dire0} and \eqref{nota0}, it is easy to calculate that in \eqref{expan2},
	\begin{equation}\notag
	k\mathbf{ d}=(k\sin\phi\cos\theta, k\sin\phi\sin\theta, -k\cos\phi)=(\alpha_{0}, -\beta_0),
	\end{equation}
	with 
	\begin{equation}\notag
		\alpha_0=\alpha:=(k\sin\phi\cos\theta, k\sin\phi\sin\theta).
	\end{equation}
	Clearly, $k\mathbf{ d}\notin\{{\bf{\xi}_n}|n\in\mathbb{Z}^2\}$ since $\phi\in(-\pi/2,\pi/2)$ and $\theta\in[0, 2\pi)$.
	Besides, from \eqref{rayley0} and \eqref{nota0}, we can know that any two vectors of $\{{\bf{\xi}_n}|n\in\mathbb{Z}^2\}$ are distinct from each other. Therefore, we can deduce the contradiction in view of \eqref{expan2} by Lemma \ref{ortho}.
	
	The proof of the uniqueness of $\eta$ is similar to the proof of Theorem \ref{inverse10} and we skip the details here to avoid repetition.
\end{proof}

Finally, we sketch the proof of the unique determination results for admissible rational polyhedral diffraction gratings as follows.

\begin{proof}[Proof of Theorem \ref{thm-grat2} (Rational case)]
	We prove by absurdity. Following the same notations and assumptions above in the proof of Theorem \ref{thm-grat0}, we suppose that there exists a corner point $\mathbf{ x}_c\in\Lambda_f$ which lies on $\partial\mathbf{G}\backslash\Lambda_g$ such that $B_h(\mathbf{ x}_c)\Subset\Omega_g$ and \eqref{B-grat} holds. Using \eqref{eq:deg cond diff}, by Theorem \ref{main-vani}, we know that $u_g(\mathbf{ x}; k, \mathbf{ d})$ satisfies
	\begin{equation}\notag
		u_g(\mathbf{ x}_c)=0,\quad \nabla u_g(\mathbf{ x}_c)=0,
	\end{equation}
which contradicts with \eqref{equi-cond2}. The uniqueness result can now be attained by a similar absurdity as stated in the proof of Theorem \ref{inverse2}. 
\end{proof}

\section*{Acknowledgement}
The work of X. Cao was supported by the Austrian Science Fund (FWF): P 32660.
The work of H. Diao was supported in part by the startup fund from Jilin University. 
The work of H. Liu was supported by the Hong Kong RGC General Research Fund (projects 12301420, 12302919 and 12301218). 
The work of J. Zou was supported by the Hong Kong RGC General Research Fund (project 14304517) and NSFC/Hong Kong RGC Joint Research Scheme 2016/17 (project N CUHK437/16).

\end{document}